\documentclass{amsart}
\usepackage{amssymb,mathtools,enumitem,hyperref,todonotes}
\usepackage{tikz}\usetikzlibrary{patterns}

\usepackage[alphabetic,initials,nobysame]{amsrefs}
\BibSpec{collection.article}{%
	+{}  {\PrintAuthors}				{author}
	+{,} { \textit}                     {title}
	+{.} { }                            {part}
	+{:} { \textit}                     {subtitle}
	+{,} { \PrintContributions}         {contribution}
	+{,} { \PrintConference}            {conference}
	+{}  {\PrintBook}                   {book}
	+{,} { }                            {booktitle}
	+{,} { }							{series}
	+{,} { }							{publisher}
	+{,} { }							{address}
	+{,} { \PrintDateB}                 {date}
	+{,} { pp.~}                        {pages}
	+{,} { }                            {status}
	+{,} { \PrintDOI}                   {doi}
	+{,} { available at \eprint}        {eprint}
	+{}  { \parenthesize}               {language}
	+{}  { \PrintTranslation}           {translation}
	+{;} { \PrintReprint}               {reprint}
	+{.} { }                            {note}
	+{.} {}                             {transition}
	+{}  {\SentenceSpace \PrintReviews} {review}
}

\newcommand{\diam}{\operatorname{diam}}
\newcommand{\dist}{\operatorname{dist}}

\newcommand{\N}{{\mathbb N}}
\newcommand{\hatc}{\hat{\mathbb{C}}}
\newcommand{\loc}{\text{loc}}
\DeclareMathOperator{\modu}{mod}
\DeclareMathOperator{\inter}{int}
\DeclareMathOperator{\area}{Area}

\newtheorem{theorem}{THEOREM}
\newtheorem{lemma}[theorem]{Lemma}
\newtheorem{proposition}[theorem]{Proposition}
\theoremstyle{remark}\newtheorem{remark}[theorem]{Remark}
\numberwithin{theorem}{section}
\numberwithin{equation}{section}

\begin{document}

\title{Exhaustions of circle domains}
\author{Dimitrios Ntalampekos}
	\address{Department of Mathematics, Aristotle University of Thessaloniki, Thessaloniki, 54152, Greece.}
\email{dntalam@math.auth.gr}
\author{Kai Rajala} 
	\address{Department of Mathematics and Statistics, University of Jyv\"askyl\"a, P.O. Box 35 (MaD), FI-40014, University of Jyv\"askyl\"a, Finland.}
	\email{kai.i.rajala@jyu.fi}

\thanks{The first author was partially supported by NSF Grant DMS-2246485. The second author was partially supported by Research Council of Finland Grant 360505.}

\begin{abstract} 
Koebe's conjecture asserts that every domain in the Riemann sphere is conformally equivalent to a circle domain. We prove that every domain $\Omega$ satisfying Koebe's conjecture admits an \textit{exhaustion}, i.e., a sequence of interior approximations by finitely connected domains, so that the associated conformal maps onto finitely connected circle domains converge to a conformal map $f$ from $\Omega$ onto a circle domain. Thus, if Koebe's conjecture is true, it can be proved by utilizing interior approximations of a domain. 

The main ingredient in the proof is the construction of \textit{quasiround} exhaustions of a given circle domain $\Omega$. In the case of such exhaustions, if $\partial \Omega$ has area zero, we show that $f$ is a M\"obius transformation. The paper builds upon a range of tools, including planar topology, Voronoi cells, classical and modern methods in (quasi)conformal mapping theory, the transboundary modulus of Schramm, and the dynamics of Schottky groups. 
\end{abstract}

\maketitle

\section{Introduction}
A domain in the Riemann sphere $\hatc$ is a \textit{circle domain} if each connected component of its boundary is a point or a circle. A long-standing problem in complex analysis is Koebe's conjecture \cite{Koe08}, predicting that every domain in $\hatc$ can be conformally mapped to a circle domain. Koebe himself established the conjecture for finitely connected domains, but it took over 70 years until the conjecture was established for countably connected domains by He--Schramm \cite{HeSch93}; an alternative argument was provided by Schramm \cite{Sch95}. The general case remains open. See also \cites{HerKos90, HilMos09, Bon16, SolVid20, EsmayliRajala:quasitripod,  KarafylliaNtalampekos:gromov_hyperbolic} for some other results related to Koebe's conjecture.

The proofs of the special cases of the conjecture follow the scheme of approximating the domain by finitely connected domains, uniformizing conformally these domains by circle domains using Koebe's theorem, and then passing to the limit. However, the limiting map will not always be a conformal map onto a circle domain, and a careful choice of the approximations is required. Thus, one of the difficulties in establishing the general case of Koebe's conjecture is that a universal approximation scheme that works for all domains is still to be found. Another subtle difficulty is that the uniformizing conformal map, if it exists, is not necessarily unique in the uncountably connected case. Therefore, one can say that there is no standard  procedure that yields the desired conformal map in a unique way. For uniqueness results related to Koebe's conjecture see \cites{HeSch94, You16, NtaYou20, Ntalampekos:rigidity_cned}.

The existence proofs by He--Schramm and Schramm proceed by approximating a given domain from the \textit{outside} by a decreasing sequence of finitely connected domains. Then Koebe’s theorem is applied to construct a sequence of conformal maps whose limit has circle domain image. Recently the second author \cite{Raj23} studied the uniformization problem by approximating a given domain from the \textit{inside} by exhaustions, i.e., increasing sequences of finitely connected subdomains, and gave an alternative proof of Koebe's conjecture in the countably connected case. 

An \textit{exhaustion} of a domain $\Omega \subset \hatc$ is a sequence of domains $\Omega_j \subset \Omega$, $j\in \N$, each bounded by finitely many disjoint 
Jordan curves in $\Omega$, such that 
$$
\Omega_j \subset \Omega_{j+1} \,\, \text{for all}\,\, j\in \N \quad \text{and}\quad \Omega = \bigcup_{j\in \N} \Omega_j. 
$$
Given an exhaustion $(\Omega_j)_{j \in \N}$, we fix distinct points $a_1,a_2,a_3 \in \Omega_1$. By Koebe's theorem, every $\Omega_j$ admits a unique conformal map
$f_j:\Omega_j \to D_j$ onto a finitely connected circle domain $D_j$ so that $f_j(a_k)=a_k$ for $k=1,2,3$. 

We say that a domain $\Omega\subset \hatc$ satisfies Koebe's conjecture if there exists a conformal map $f$ from $\Omega$ onto a circle domain. 
\begin{theorem}\label{theorem:koebe_domains}
    Let $\Omega\subset \hatc$ be a domain that satisfies Koebe's conjecture. Then there is an exhaustion $(\Omega_j)_{j \in \N}$ of $\Omega$ and a circle domain $D\subset \hatc$ so that $(f_j)_{j \in \N}$ converges locally uniformly in $\Omega$ to a conformal homeomorphism $f\colon \Omega \to D$.   
\end{theorem}

Thus, if Koebe's conjecture is true, then it can be proved by using exhaustions. The method of using exhaustions also appears in uniformizing a domain by horizontal slit domains. Namely, if $\Omega$ is any domain in $\mathbb C$ and $(\Omega_j)_{j \in \N}$ is \textit{any} exhaustion of $\Omega$, then the conformal maps $f_j$ from $\Omega_j$ onto finitely connected slit domains, normalized appropriately, converge, after passing to a subsequence, to a conformal map from $\Omega$ onto a slit domain \cite{Cou50}*{Theorem 2.1, p.~54}. The reason is that this conformal map arises as the solution to a certain Dirichlet energy minimization problem.

In sharp contrast to that result, the conclusion of Theorem \ref{theorem:koebe_domains} is not true for \textit{any} exhaustion, as shown by the second author \cite{Raj23}*{Theorems 1.1 and 1.2}. Hence, a careful choice of $(\Omega_j)_{j \in \N}$ is required. 

Theorem \ref{theorem:koebe_domains} is an immediate consequence of the next theorem.

\begin{theorem}\label{mainthm}
For every circle domain $\Omega \subset \hatc$ there is an exhaustion $(\Omega_j)_{j \in \N}$ of $\Omega$ and a circle domain $D\subset \hatc$ so that $(f_j)_{j \in \N}$ converges locally uniformly in $\Omega$ to a conformal homeomorphism $f\colon\Omega \to D$.  In addition, if $\partial \Omega$ has area zero, then $D=\Omega$ and $f$ is the identity map.
\end{theorem} 

The proof of Theorem \ref{mainthm} is based on a variety of ingredients, including planar topology, Voronoi cells, classical conformal mapping theory, the transboundary modulus of Schramm, modern techniques in quasiconformal mapping theory, and the dynamics of Schottky groups generated by reflections in disks. We describe the main steps for the proof of Theorem \ref{mainthm}.

(1) We prove that every circle domain $\Omega$ admits a \textit{quasiround} exhaustion $(\Omega_j)_{j \in \N}$. That is, there exists $K\geq 1$ such that for each $j\in \N$, each complementary component $p$ of $\Omega_j$ is $K$-quasiround in the following sense: there are $a\in p$ and $r>0$ such that
$$\overline{\mathbb D}(a,r)\subset p\subset \overline{\mathbb D}(a,Kr).$$
We prove the existence of quasiround exhaustions with $K=43$ in Section \ref{section:exhaustion}. The proof is technical and involves an efficient covering of $\hatc\setminus \Omega$ by essentially disjoint balls and then the construction of Voronoi cells corresponding to this collection of balls. The Voronoi cells are carefully modified, with the help of some intricate topological tools such as the Moore decomposition theorem, in order to produce the complementary components of a quasiround domain $\Omega_j$ that approximates $\Omega$ from inside. We include a detailed proof outline in Section \ref{section:exhaustion_outline} before giving the proof.

It may be desirable to have $K$ be arbitrarily close to $1$ so that the approximating domains $\Omega_j$ resemble the circle domain $\Omega$ as much as possible. In Section \ref{section:example} we show with a simple example that this is not possible, in general.

(2) Let $(\Omega_j)_{j \in \N}$ be an exhaustion of a circle domain $\Omega$ and $f_j\colon \Omega_j\to D_j$, $j\in \N$, be a conformal map provided by Koebe's uniformization theorem (for finitely connected domains) that fixes three distinct points of $\Omega_1$. After passing to a subsequence, by normality criteria we see that $(f_j)_{j \in \N}$ converges locally uniformly in $\Omega$ to a holomorphic map $f$ on $\Omega$. By a classical theorem of Hurwitz, $f$ is a conformal map from $\Omega$ onto some domain $D$. In general $D$ is not a circle domain. We show that if the exhaustion $(\Omega_j)_{j \in \N}$ is quasiround, then $D$ is indeed a circle domain. This statement is proved in Section \ref{section:convergence}. 

The main task in the proof is to establish a certain boundary equicontinuity property of the maps $f_j$. This is done by utilizing a powerful modern tool in conformal mapping theory, the \textit{transboundary modulus} of curve families, which was introduced in a groundbreaking work of Schramm \cite{Sch95}. We use the geometry of circle domains and quasiround domains to establish bounds of transboundary modulus that resemble well-known bounds of the \emph{classical (conformal) modulus}. Similar estimates appear in \cites{Sch95, Bon11, Raj23}. We include a detailed proof outline in Section \ref{section:convergence_outline}. 

This result completes the proof of the first part of Theorem \ref{mainthm}. The next two steps concern the uniqueness statement in the second part of the theorem. 

(3) The uniqueness statement of Theorem \ref{mainthm} is related to the problem of conformal {rigidity} of circle domains. A circle domain is \textit{conformally rigid} if every conformal map onto another circle domain is the restriction of a M\"obius transformation. Sufficient conditions for the rigidity of a circle domain have been provided in \cites{HeSch94, You16, NtaYou20, Ntalampekos:rigidity_cned}, but a characterization of rigid domains is yet to be discovered. Therefore, if the domain $\Omega$ in Theorem \ref{mainthm} were rigid, then the desired conclusion would follow immediately. Nevertheless, the condition that the area of $\partial \Omega$ vanishes, which is the only restriction in Theorem \ref{mainthm}, is \textit{not} sufficient to yield rigidity, as elementary examples show. In order to prove the last statement of Theorem \ref{mainthm}, we have to rely on some \textit{additional regularity} that the particular limiting conformal map $f\colon \Omega\to D$ has.

Specifically, if $\partial \Omega$ has area zero, we prove that the map $g=f^{-1}\colon D\to \Omega$ satisfies a type of \textit{transboundary upper gradient inequality}, as stated in Section \ref{section:regularity}. This inequality is motivated by the transboundary modulus of Schramm. Maps that satisfy such inequalities have been recently studied by the first author in \cite{Ntalampekos:uniformization_packing} in connection with Koebe's conjecture and with the uniformization of Sierpi\'nski carpets. The transboundary upper gradient inequality implies that the map $g$ is absolutely continuous, in a sense, not only on curves inside the domain $D$, but also on curves that travel through complementary components of $D$. It is precisely this type of regularity that we need to deduce that $g$ is a M\"obius transformation, as explained in the next step.

(4) We finally prove in Section \ref{section:uniqueness} that any conformal map $g\colon D\to \Omega$ between arbitrary circle domains that satisfies the transboundary upper gradient inequality of Section \ref{section:regularity} is the restriction of a M\"obius transformation of $\hatc$; see Theorem \ref{theorem:uniqueness}. To do so, we first use the transboundary upper gradient inequality to extend $g$ to a homeomorphism between the closures of $D$ and $\Omega$; this relies on results from \cite{Ntalampekos:uniformization_packing}. Then the extended map $g$ can be further extended to a homeomorphism of the entire sphere by using repeated reflections along the boundary circles of $D$ and $\Omega$; similar arguments have appeared in \cites{HeSch94,BonkKleinerMerenkov:schottky,NtaYou20,Ntalampekos:rigidity_cned}. The goal now is to show that this extended map is conformal, and thus it is a M\"obius transformation. 

Conformality is obtained by a variation of the metric definition of quasiconformality proved by the first author \cite{Ntalampekos:metric_definition_qc}. Specifically, if the \textit{eccentric distortion} of a homeomorphism $g$ between planar domains is $1$ off a small set, then $g$ is a conformal map. We discuss this result in Section \ref{section:definition_qc} and we prove a slight variant of the main theorem of \cite{Ntalampekos:metric_definition_qc} that we need for our purposes. We note that the work \cite{Ntalampekos:metric_definition_qc} that we rely on can be regarded as the culmination of a series of works regarding the metric definition of quasiconformality in Euclidean space \cites{Gehring:Rings, HeinonenKoskela:liminf, KallunkiKoskela:quasiconformal, KallunkiKoskela:exceptional2}.

\subsection*{Acknowledgments}
The statement of Theorem \ref{mainthm} was a problem posed to us by Dennis Sullivan, whom we thank for various conversations on the topic.

\section{Construction of quasiround exhaustions}\label{section:exhaustion}
Given a domain $G \subset \hatc$, we denote by $\mathcal{C}(G)$ the collection of connected components of
$\hatc \setminus G$. Moreover, $\mathcal{C}(G)=\mathcal{C}_N(G) \cup \mathcal{C}_P(G)$, where $\mathcal{C}_N(G)$ is the collection of components with positive diameter ($N$ stands for non-degenerate), and
$\mathcal{C}_P(G)$ is the collection of point-components ($P$ stands for point). If $(G_j)_{j \in \N}$ is an exhaustion of $G$, $\bar{p} \in \mathcal{C}(G)$, and $j \geq 1$, we denote by $p_j(\bar{p})$ the element of
$\mathcal{C}(G_j)$ containing $\bar{p}$. 

We denote $\hat{G}=\hatc / \sim$, where
$$
x \sim y \text{ if either } x=y \in G \text{ or } x,y \in p \text{ for some } p \in \mathcal{C}(G). 
$$
The corresponding quotient map is $\pi_G:\hatc \to \hat{G}$. Identifying each $x \in G$ and $p \in \mathcal{C}(G)$ with
$\pi_G(x)$ and $\pi_G(p)$, respectively, we have 
$$
\hat{G}=G \cup \mathcal{C}(G). 
$$ 
A homeomorphism $f:G \to G'$ has a homeomorphic extension $\hat{f}:\hat{G} \to \hat{G'}$; see \cite{NtaYou20}*{Section 3} for a detailed discussion. By Moore's theorem \cite{Moore:theorem}, the quotient space $\hat G$ is homeomorphic to $\hatc$.

All distances below refer to the Euclidean metric of $\mathbb C$. In what follows, if $D=\mathbb{D}(z,r)$ is a disk and $\tau>0$ then $\tau D=\mathbb{D}(z,\tau r)$. Moreover, $\mathbb{S}(z,r)$ is the boundary circle of $\mathbb{D}(z, r)$. 

For $K\geq 1$ we say that a set $A \subset \mathbb{C}$ is \textit{$K$-quasiround}, if there are $z_A \in \mathbb{C}$ and $r_A>0$ such that  
$$
\overline{\mathbb{D}}(z_A,r_A) \subset A \subset \overline{\mathbb{D}}(z_A,Kr_A).
$$
Moreover, we say that a domain $G \subset \hatc$ is $K$-quasiround if every $p \in \mathcal{C}_N(G)$ is $K$-quasiround, and that a sequence of domains $G_j \subset \hatc$, $j\in \N$, is $K$-quasiround if every $G_j$ is $K$-quasiround. 

\begin{theorem}\label{theorem:quasiround}
    Every circle domain $\Omega \subset \hatc$ with $\infty \in \Omega$ has a $43$-quasiround exhaustion $(\Omega_j)_{j \in \N}$. 
\end{theorem}


\subsection{Outline of the proof of Theorem \ref{theorem:quasiround}}\label{section:exhaustion_outline}
We briefly discuss the main ideas of the proof of Theorem \ref{theorem:quasiround} given below. We are given a circle domain $\Omega$, and our task is to find a quasiround sequence of domains $\Omega_j$, $j\in \N$, each bounded by finitely many Jordan curves in $\Omega$, such that the complements $Z_j =\hatc \setminus \Omega_j$, $j\in \N$, form a nested sequence of compact sets for which $\bigcap_{j=1}^\infty Z_j =Z:=\hatc \setminus \Omega$. 

We first consider the special case of circle domains $\Omega$ for which $Z$ is totally disconnected, i.e., domains without non-degenerate complementary components. We start with a domain $\Omega_0 \subset \Omega$ that is the complement of a large closed disk whose boundary lies in $\Omega$. We fix $j \geq 1$ and assume that the desired $K$-quasiround domain $\Omega_{j-1}$ has been constructed (where $K=43$). Recall that $\partial \Omega_{j-1} \subset \Omega$. Our goal is to construct a $K$-quasiround domain $\Omega_j$ bounded by finitely many Jordan curves in $\Omega$ such that $\Omega_{j-1}\subset \Omega_j\subset \Omega$ and $\hatc \setminus \Omega_j$ is contained in the $(1/j)$-neighborhood of $Z=\hatc\setminus \Omega$. For this purpose, we fix a number $\delta>0$ that is smaller than both $\frac{1}{j}$ (to guarantee the convergence of $\Omega_j$ to $\Omega$) and the distance between $Z$ and $\Omega_{j-1}$ (to guarantee that $\Omega_j \supset \Omega_{j-1}$).  

We apply the $5r$-covering lemma \cite{Heinonen:metric}*{Theorem 1.2} to obtain a finite collection $\mathcal{B}$ of disjoint closed disks $B$ of common radius $\delta$ such that the union $\bigcup_{B \in \mathcal{B}}5B$  of the disks $5B$ with the same centers and radius $5\delta$ covers the set $\hatc\setminus \overline{\Omega_{j-1}}\supset Z$. We start the construction by defining for every $B \in \mathcal{B}$ an open set $V_B$, consisting of points $z \in \hatc \setminus \overline{\Omega_{j-1}}$ that are closer to $B$ than to any other $B' \in \mathcal{B}$. It follows that $B \subset V_B \subset 5B$, so that each $V_B$ is quasiround. We also discard from the collection the sets $V_B$ that are close to $\partial \Omega_{j-1}$; the remaining sets still cover $Z$ (see Lemma \ref{lemma:u}). The sets $V_B$, $B\in \mathcal B$, are actually \textit{Voronoi cells}. 

One can show (see Lemma \ref{lemma:u'}) that the sets $V_B$ possess most of the properties required from the components of $Z_j=\hatc \setminus \Omega_j$: they are quasiround, they have pairwise disjoint interiors, and their boundaries are Jordan curves. On the other hand, they have two properties that are not desired: 
\begin{enumerate} 
\item their \emph{closures} are not pairwise 
disjoint, and 
\item \label{too} the boundaries $\partial V_B$ may intersect $Z$. 
\end{enumerate}  

The remaining task for the proof of Theorem \ref{theorem:quasiround} is to find deformations of the sets $\overline{V_B}$ into new closed Jordan regions $q_B$ that still possess the good properties discussed above and do not have the undesirable Property \ref{too}; see Lemma \ref{lemma:qb}. After such sets $q_B$ are found, it is a simple matter to apply another deformation to find a collection of slightly smaller pairwise disjoint closed Jordan regions $p_B\subset q_B$, whose complement is defined to be the new domain $\Omega_j$. This domain will have all the required properties. 

In order to construct the sets $q_B$ in Lemma \ref{lemma:qb} whose boundaries avoid the set $Z$, we prove a relevant Perturbation Lemma \ref{lemma:moore}. This lemma gives a way to find the desired deformations of the sets $\overline{V_B}$ by considering perturbations of their boundaries $\bigcup_{B\in \mathcal B}\partial V_B$ in a suitable \emph{decomposition space}. 

To prove the Perturbation Lemma \ref{lemma:moore}, we first cover the totally disconnected set $Z$ with finitely many open sets $A_{i}$ with small diameters and boundaries in $\Omega$, and form the decomposition space $W$ by shrinking each $A_i$ to a point. Since the set $Z$ is covered by the sets $A_i$, it is easier to consider Property \ref{too} in the space $W$ than in the original setting; now there are only finitely many points to avoid, namely the projections of the sets $A_i$. After applying Moore's decomposition theorem, we may assume that $W=\hatc$. We perturb the union of the boundaries $\bigcup_{B\in \mathcal B}\partial V_B$ by applying \emph{translations} to its image in $W$. An argument involving the Baire category theorem guarantees that some of those perturbations avoid the sets $A_i$, thus giving us the desired sets $q_B$; see Figure \ref{fig:qb} for an illustration. This concludes the outline of the proof of Theorem \ref{theorem:quasiround} in the case where $Z$ is totally disconnected.

In the general case, the complement of $\Omega$ consists of point-components and a finite or infinite collection of disks $D_k$. We fix a finite collection of large disks $D_k$. Next we consider a finite collection of pairwise disjoint disks $B$ (provided as above by the $5r$-covering lemma) that are disjoint from the chosen large disks $D_k$ and such that the union of the disks $5B$ and of a certain enlargement of the chosen disks $D_k$ covers the set $\hatc\setminus \overline{\Omega_{j-1}}$. We let $\mathcal U$ be the collection that contains the disks $B$ and the large disks $D_k$. We now construct a Voronoi cell decomposition as above that corresponds to the collection of disks $\mathcal U$. The argument then proceeds as above in principle, in order to deform the Voronoi cells into disjoint closed Jordan regions $p_B$ that are $K$-quasiround, cover the set $Z$, and have boundaries in $\Omega$. There is an additional complication in some claims because one has to consider separately Voronoi cells that arise from the disks $B$ and from the disks $D_k$.


\subsection{Proof of Theorem \ref{theorem:quasiround}} 

The collection $\mathcal{C}_N(\Omega)$ of non-degenerate components of $\hatc\setminus \Omega$ is empty or consists of finitely or countably many disks  
$$
D_k=\overline{\mathbb{D}}(z_k,r_k), \quad \text{where} \quad r_1 \geq r_2 \geq r_3 \cdots. 
$$ 

We will construct the required quasiround exhaustion of $\Omega$ inductively. Let $\Omega_0$ be the complement of a large closed disk that contains $\hatc\setminus \Omega$ in its interior. Assume that $j \geq 1$ and that we have defined a domain $\Omega_{j-1} \subset \overline{\Omega_{j-1}} 
\subset \Omega$ that is bounded by finitely many disjoint Jordan curves in $\Omega$ so that $\hatc\setminus \overline{\Omega_{j-1}}$ is bounded. Theorem \ref{theorem:quasiround} follows if we find a $43$-quasiround domain $\Omega_j$ that is bounded by finitely many disjoint Jordan curves in $\Omega$ and satisfies 
\begin{equation} \label{quasiroundinclusion}
\Omega_{j-1} \subset \Omega_j \subset \Omega \quad \text{and} \quad \hatc \setminus \Omega_j \subset N_{1/j}(\hatc \setminus \Omega),  
\end{equation}
where $N_{1/j}(A)$ is the open $(1/j)$-neighborhood of $A$. 
We define  
\begin{equation}\label{defdelta}
\delta=\frac{\min\Big\{\frac{1}{j},\dist(\partial \Omega,\partial \Omega_{j-1})\Big\}}{100}>0  
\end{equation}
and  
$$
\mathcal{C}_L(\Omega)=\{D_1,\ldots, D_\alpha\} \subset \mathcal{C}_N(\Omega), 
$$ 
where $\alpha$ is the largest index for which $r_\alpha \geq \delta/4$.

Next, we denote  
$$
U=\hatc \setminus \big(\overline{\Omega_{j-1}} \cup \big(\bigcup_{\beta=1}^\alpha 
(1+2\delta/r_{\beta})D_\beta \big)\big).      
$$
Observe that the radius of the disk $(1+2\delta/r_{\beta})D_\beta$ is equal to $r_{\beta}+2\delta$ and that each such disk is disjoint from $\overline{\Omega_{j-1}}$. By the $5r$-covering lemma \cite{Heinonen:metric}*{Theorem 1.2}, since $U$ is bounded, there exists a finite collection $\mathcal{B}'$ of pairwise disjoint disks of radius $\delta$ centered at $U$, so that  
\begin{equation*}
U \subset \bigcup_{B\in \mathcal B'} 5B. 
\end{equation*}
We let $\mathcal{U}'=\mathcal{C}_L(\Omega) \cup \mathcal{B}'$ and for $B \in \mathcal{U}'$ we define
$$
V_B =\{z \in \hatc \setminus \overline{\Omega_{j-1}}: \dist(z,B)<\dist(z,B') \,\, \text{for every}\,\, B' \in \mathcal{U}', \, B' \neq B\}. 
$$
See Figure \ref{fig:vb} for an illustration. 

\begin{lemma}[Properties of $\mathcal U'$]\label{lemma:u'}
    The following statements are true.
    \begin{enumerate}[label=\normalfont (\arabic*)]
        \item\label{lemma:u':b} The closed disks $B$, $B\in \mathcal U'$, are pairwise disjoint.
        \item\label{lemma:u':vb} The sets $V_B$, $B\in \mathcal U'$, are open and pairwise disjoint.
        \item\label{lemma:u':boundary} For every $B\in \mathcal U'$, 
        $$\{z\notin \overline{\Omega_{j-1}} \cup V_B: \textrm{$ \dist(z,B)\leq \dist(z,B')$ for all $B'\in \mathcal U'$}\}= \partial V_B \setminus \overline{\Omega_{j-1}}.$$
        \item\label{lemma:u':starlike} For every $B\in \mathcal U'$, if $\overline{V_B}\cap \overline{\Omega_{j-1}}=\emptyset$, then $V_B$ is star-like with respect to the center of $B$.        
        \item\label{lemma:u':jordan} For every $B\in \mathcal U'$, if $\overline{V_B}\cap \overline{\Omega_{j-1}}=\emptyset$, then $V_B$ is a Jordan region. 
        \item\label{lemma:u':union} $\hatc\setminus\overline{\Omega_{j-1}} \subset \bigcup_{B\in \mathcal U'} \overline{V_B}$.
        \item\label{lemma:u':inclusions} For each $B=\overline{\mathbb D} (z_B,r_B)\in \mathcal U'$ we have 
    $$V_B \subset \overline{\mathbb{D}}(z_B,r_B+4\delta) \subset \overline{\mathbb{D}}(z_B,17r_B).$$
        Moreover, if $\overline{V_B}\cap \overline{\Omega_{j-1}}=\emptyset$, then $\overline{\mathbb D} (z_B,r_B)\subset V_B$.
    \end{enumerate}
\end{lemma}

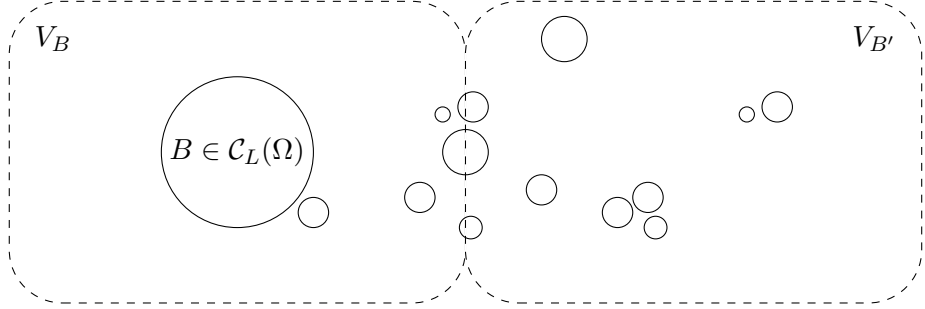
\begin{figure}
    \centering
    
\begin{tikzpicture}
\draw (0,0) circle (1cm);
\draw[rounded corners=20pt, dashed] (-3,-2)--(3,-2)-- (3,2)--(-3,2)--cycle;

\draw[color=white] (3,2)--(3,-2);

\draw (4.3,1.5) circle (0.3cm);
\draw (5.5,-1) circle (0.15cm);
\draw (6.7,0.5) circle (0.1cm);
\draw (7.1,0.6) circle (0.2cm);
\draw (5.4,-0.6) circle (0.2cm);
\draw (5,-0.8) circle (0.2cm);

\draw[rounded corners=20pt, dashed] (3,-2)--(9,-2)-- (9,2)--(3,2)--cycle;

\draw (3,0) circle (0.3cm);
\draw (3.07,-1) circle (0.15cm);
\draw (2.7,0.5) circle (0.1cm);
\draw (3.1,0.6) circle (0.2cm);
\draw (2.4,-0.6) circle (0.2cm);
\draw (1,-0.8) circle (0.2cm);
\draw (4,-0.5) circle (0.2cm);

\node[anchor=north west] at (-2.8,1.8) {$V_B$};
\node at (0,0) {$B\in \mathcal C_{L}(\Omega)$};
\node[anchor=north east] at (8.8,1.8) {$V_{B'}$};
\end{tikzpicture}

    \caption{A region $V_B$ corresponding to some $B\in \mathcal C_L(\Omega)$ and a region $V_B'$ corresponding to some $B'\in \mathcal B'$. The shown disks are components of $\hatc\setminus \Omega$; $B'$ is not visible in the figure.}
    \label{fig:vb}
\end{figure}

\begin{proof}
    The disks of $\mathcal B'$ are pairwise disjoint and the disks of $\mathcal C_L(\Omega)$ are pairwise disjoint. Also, since each $B\in \mathcal B'$ is centered at $U$ and has radius $\delta$, its distance from all disks of $\mathcal C_L(\Omega)$ is at least $\delta$. This shows that the collection $\mathcal U'$ is disjointed, as required in \ref{lemma:u':b}. Part \ref{lemma:u':vb} follows immediately from the definition of $V_B$.

    Let $B=\overline{\mathbb D}(z_B,r_B)\in \mathcal U'$. For $\theta\in [0,2\pi)$ and $t\geq 0$, let $w_t=z_B+te^{i\theta}$ and $r_t=\dist(w_t,B)$.  Suppose that $w_t \notin \overline{\Omega_{j-1}}\cup V_B$, and $\dist(w_t,B)\leq \dist(w_t,B')$ for all $B'\in \mathcal U'$. Since $w_t\notin V_B$, there exists $B''\in \mathcal U'$ with $B''\neq B$ such that
    $$r_t=\dist(w_t,B) =\dist(w_t,B'').$$
    Since $B$ and $B''$ are disjoint, we must have $r_t>0$, so $t>r_B$. The disk $\mathbb D(w_t,r_t)$ is externally tangent to both $B$ and $B''$, and it is disjoint from all $B'\in \mathcal U'$. For $s\in (r_B,t)$ the disk $\mathbb D(w_s,r_s)$ is a strict subset of $\mathbb D(w_t,r_t)$ and is tangent only to $B$ and not to any $B'\in \mathcal U'\setminus \{B\}$. Hence, we have $\dist(w_s,B)<\dist(w_s,B')$ for all $B'\in \mathcal U'\setminus \{B\}$. Also, if $s$ is sufficiently close to $t$, then $w_s\in \hatc\setminus \overline{\Omega_{j-1}}$.  This implies that $w_s\in V_B$ for all $s<t$ near $t$. Thus, $w_t\in \partial V_B$. Conversely, if $w_t\in \partial V_B\setminus \overline{\Omega_{j-1}}$, then by the definition of $V_B$ we must have $\dist(z,B)\leq \dist(z,B')$ for all $B'\in \mathcal U'$. This proves \ref{lemma:u':boundary}. The proof also shows that if $w_t\in \partial V_B\setminus \overline{\Omega_{j-1}}$, then
     \begin{align}\label{lemma:u':boundary:conclusion}
        \textrm{$\dist(w_s,B)<\dist(w_s,B')$ for $s\in [0,t)$ and $B'\in \mathcal U\setminus \{B\}$. }
    \end{align}
    
    Suppose that  $\overline{V_B}\cap \overline{\Omega_{j-1}}=\emptyset$. If $w_t\in \partial V_B$, then by  \eqref{lemma:u':boundary:conclusion}, we have $w_s\in V_B$ for all $s<t$ near $t$.  If there exists $s\in [0,t)$ with $w_s \notin V_B$, then by \eqref{lemma:u':boundary:conclusion} we must have $w_s\in \overline{\Omega_{j-1}}$. Hence,  there exists $s'\in [s,t)$ with $w_{s'} \in \overline{V_B}\cap \overline{\Omega_{j-1}}$. This is a contradiction. Therefore, the segment $\{w_s: 0\leq s<t\}$ is contained in $V_B$ and $V_B$ is star-like with respect to $z_B$ as claimed in \ref{lemma:u':starlike}.

    For \ref{lemma:u':union}, 
    let $z\in \hatc\setminus\overline{\Omega_{j-1}}$ and consider $B\in \mathcal U'$ that minimizes $\dist(z,B')$ over all $B'\in \mathcal U'$. Then $\dist(z,B)\leq \dist(z,B')$ for all $B'\in \mathcal U'$. By \ref{lemma:u':boundary}, $z\in \overline{V_B}$.
    
    Next, for \ref{lemma:u':inclusions}, observe that if $z\in \hatc\setminus \overline{\Omega_{j-1}}$, then either $z\in (1+2\delta/r_{\beta})D_{\beta}$ for some $\beta\in\{1,\dots,\alpha\}$, in which case $\dist(z,D_{\beta})<2\delta$, or $z\in U$ so $z\in 5B =\overline{\mathbb D}(z_B,5\delta)$ for some  $B\in \mathcal B'$. In any case,
\begin{align*}
    \textrm{if $z\in \hatc \setminus \overline{\Omega_{j-1}}$, then $\dist(z,B)<4\delta$ for some $B\in \mathcal U'$.}
\end{align*}
    Now, let $B\in \mathcal U'$ and $z\in V_B$. Since $z\in \hatc \setminus \overline{\Omega_{j-1}}$, by the above we have $\dist(z,B')<4\delta$ for some $B'\in \mathcal U'$. The definition of $V_B$ implies that $\dist(z,B)\leq \dist(z,B')<4\delta$. This implies the first inclusion claimed in \ref{lemma:u':inclusions}. Moreover, since the radius of every $B\in \mathcal U'$ is at least $\delta/4$, the second inclusion holds as well.  
    
    Suppose that $\overline{V_B}\cap \overline{\Omega_{j-1}}=\emptyset$, as the last part of \ref{lemma:u':inclusions}. If $B \not\subset V_B$, since $z_B\in V_B$ by \ref{lemma:u':starlike}, we have $B\cap \partial V_B\neq \emptyset$. If $z\in B\cap \partial V_B$, we have $z\in \hatc\setminus \overline{\Omega_{j-1}}$ and $\dist(z,B)=0 <\dist(z,B')$ for all $B'\in \mathcal U'\setminus \{B\}$. This implies that $z\in V_B$, a contradiction. Therefore $B\subset V_B$, as desired. 

    For \ref{lemma:u':jordan}, observe that $V_B$ is an open set that is simply connected by part \ref{lemma:u':starlike}. Since $V_B$ is bounded by part \ref{lemma:u':inclusions}, $\partial V_B$ is a continuum. Since $V_B$ is star-like and bounded, $\mathbb C \setminus \overline{V_B}$ is connected. Hence $\mathbb C\setminus \partial V_B$ has two connected components. If $z\in \partial V_B \subset \hatc\setminus \overline{\Omega_{j-1}}$, by the definition of $V_B$ there exists $B'\in \mathcal U'\setminus \{B\}$, such that 
    $$\dist(z, B)=\dist(z,B') \leq \dist(z,B'')$$
    for all $B''\in \mathcal U'$. By part \ref{lemma:u':boundary}, we have $z\in \partial V_{B'}$. By \ref{lemma:u':starlike} there exists a line segment connecting $z_B$ to $z$ and whose interior is contained in $V_B$. 
    By \eqref{lemma:u':boundary:conclusion}, there exists a line segment connecting $z$ to a point of $V_{B'}$ and whose interior is contained in $V_{B'}\subset \mathbb C\setminus \overline{V_B}$. This shows that each point of the continuum $\partial V_B$ is accessible from both complementary components of $\partial V_B$. The inverse of the Jordan curve theorem \cite{Kuratowski:topology}*{Theorem 61.II.12, p.~518} implies that $\partial V_B$ is a Jordan curve.
\end{proof}

Lemma \ref{lemma:u'} implies that the sets $V_B$, $B\in \mathcal U'$, satisfy several conditions that are expected to be true for the complementary components of the domain $\Omega_j$, which we have set out to construct: they are quasiround, they have pairwise disjoint interiors, they cover the set $\hatc\setminus \Omega$, and with some exceptions they are Jordan regions. However, the complementary components of $\Omega_j$ are expected to have stronger properties: their \textit{closures} must be {disjoint}, their boundaries must be \textit{contained} in $\Omega$, and their boundaries must be \textit{Jordan regions}. We amend these issues one by one. 

First, we discard some regions $V_B$, $B\in \mathcal U'$, that might fail to be Jordan regions. Specifically, we let 
\begin{equation*} 
\mathcal{U}=\{B \in \mathcal{U}': \,\overline{V_B} \cap 
N_{10\delta}(\Omega_{j-1})=\emptyset \}. 
\end{equation*}
Note that by Lemma \ref{lemma:u'} \ref{lemma:u':jordan}, all sets $V_B$, where $B\in \mathcal U$, are Jordan regions. In the next lemma we show that their union still covers a neighborhood of the set $\hatc\setminus \Omega$.

\begin{lemma}[Properties of $\mathcal U$]\label{lemma:u}
We have $\mathcal C_L(\Omega) \subset \mathcal U$ and $$N_{10\delta}(\hatc \setminus \Omega) \subset \bigcup_{B \in \mathcal{U}} \overline{V_B}.$$
\end{lemma}

\begin{proof}
For each $B\in \mathcal C_L(\Omega)$ we have $V_B\subset N_{4\delta}(B)$ by Lemma \ref{lemma:u'} \ref{lemma:u':inclusions}. Since $B\subset \hatc\setminus \Omega$, we see that $V_B\subset N_{4\delta}(\hatc \setminus \Omega)$.  By the definition of $\delta$ in \eqref{defdelta}, we have
\begin{equation}\label{fara}
\dist(N_{10\delta}(\Omega_{j-1}),N_{10\delta}(\hatc\setminus \Omega))\geq 80\delta.
\end{equation}
This shows that $\overline{V_B}$ is disjoint from $N_{10\delta}(\Omega_{j-1})$, so $B\in \mathcal U$ and $\mathcal C_L(\Omega) \subset \mathcal U$.

Next, the sets $\overline{V_B}$, $B \in \mathcal{U'}$, cover $\hatc \setminus \overline{\Omega_{j-1}}$ by Lemma \ref{lemma:u'} \ref{lemma:u':union}. By \eqref{fara}, $N_{10\delta}(\hatc\setminus \Omega)$ is also covered by $\overline{V_B}$, $B \in \mathcal{U'}$.
By the first inclusion in Lemma \ref{lemma:u'} \ref{lemma:u':inclusions}, if $B \in \mathcal{B}'$, then $\overline{V_B}$ has diameter at most $10\delta$, so it cannot intersect both sets in \eqref{fara}. Therefore, each point of $N_{10\delta}(\hatc\setminus \Omega)$ is contained in a set $\overline{V_B}$ such that $B\in \mathcal B'$ and $\overline{V_B}\cap N_{10\delta}(\Omega_{j-1})=\emptyset$ or $B\in \mathcal C_L(\Omega)\subset \mathcal U$. This completes the proof. 
\end{proof}

Next, we wish to appropriately deform the Jordan regions $V_B$, $B\in \mathcal U$, into some regions whose boundaries lie in $\Omega$. In order to do so, we need the following general topological lemma, which is a consequence of Moore's theorem \cite{Moore:theorem}. 

\begin{lemma}[Perturbation lemma]\label{lemma:moore}
   Let $\varepsilon>0$, $m\in \N$, and $A_1,\dots,A_m\subset \mathbb C$ be pairwise disjoint and non-separating continua with $\diam A_i<\varepsilon$ for each $i\in \{1,\dots,m\}$. Then for each closed nowhere dense set $E\subset \mathbb C$  there exists a homeomorphism $\phi \colon \mathbb C\to \mathbb C$ such that 
   $$\phi(E) \cap  \Big(\bigcup_{i=1}^m A_i\Big)=\emptyset\quad \textrm{and}\quad  \sup_{z\in \mathbb C}|\phi(z)-z|<\varepsilon.$$  
   Moreover, if $F\subset \mathbb C \setminus \bigcup_{i=1}^m A_i$ is a compact set with $E\cap F=\emptyset$, then we may also have $\phi(E)\cap F=\emptyset$.
\end{lemma}

\begin{proof}
    By Moore's theorem \cite{Daverman:decompositions}*{Theorem 25.1}, the decomposition of $\hatc$ induced by the sets $A_i$, $i\in \{1,\dots,m\}$, and singleton points $z\in \hatc\setminus \bigcup_{i=1}^m A_i$ is \textit{strongly shrinkable}. In particular, this implies that there exists a set of points $\{p_1,\dots,p_m\}\subset \hatc$ and a continuous and surjective map $\pi\colon \mathbb \hatc\to \mathbb \hatc$ such that $\pi$ is the identity map outside a disk $\mathbb D(0,R)$ that contains $\bigcup_{i=1}^m A_i$, $\pi$ is injective outside $\bigcup_{i=1}^m A_i$, and $\pi^{-1}(p_i)=A_i$, $i\in \{1,\dots,m\}$.  Moreover, there exists a sequence of homeomorphisms $\pi_k\colon \hatc\to \hatc$, $k\in \N$, that converge uniformly to $\pi$, and $\pi_k$ is also the identity map outside $\mathbb D(0,R)$ for each $k\in \N$; see \cite{Daverman:decompositions}*{Section 5} for properties of strongly shrinkable decompositions. 
    
    By a compactness argument, there exists $\eta>0$ such that if $A\subset \mathbb C$ is a compact set and $\diam A\leq \eta$, then $\diam \pi^{-1}(A)<\varepsilon$. By uniform convergence, there exists $k_1\in \N$ such that if $A$ is compact and $\diam A\leq \eta$, then we also have $\diam \pi_k^{-1}(A)<\varepsilon$ for all $k\geq k_1$. 

    Let $F\subset \mathbb C\setminus \bigcup_{i=1}^m A_i$ be a compact set that is disjoint from $E$. Then $\pi(E)\cap \pi(F)=\emptyset$, since $\pi$ is injective in $\hatc\setminus \bigcup_{i=1}^m A_i$. The set $\pi(E)$ is nowhere dense. For $i\in \{1,\dots,m\}$, let 
    $$S_i=\{ z\in \overline{\mathbb D}(0,\eta): p_i\in \pi(E)-z\}=  \overline{\mathbb D}(0,\eta)\cap  (\pi(E)-p_i).$$
    The set $\bigcup_{i=1}^m S_i$ is nowhere dense. Hence, there exists $z_0\in \overline{\mathbb D}(0,\eta)$, arbitrarily close to $0$, such that the set $\pi(E)-z_0$ is disjoint from $\{p_1,\dots,p_m\}$. If $z_0$ is sufficiently close to $0$, then we may have that $\pi(E)-z_0$ is also disjoint from the compact set $\pi(F)$. By uniform convergence, there exists $k_2\in \N$  such that for $k\geq k_2$ the set $\pi_k(E)-z_0$ is disjoint from $\bigcup_{i=1}^m \pi_k(A_i)$ and from $\pi_k(F)$. 
    
    Fix $k\geq \max\{k_1,k_2\}$ and let $\psi(z)=z-z_0$ on $\hatc$. We define $\phi= \pi_k^{-1} \circ \psi\circ \pi_k$, which is a homeomorphism of $\hatc$ fixing $\infty$. By construction, $\phi(E)$ is disjoint from $\bigcup_{i=1}^m A_i$ and from $F$. For $z\in \mathbb C$ and $w=\pi_k(z)$, we have $|\psi(w)-w|=|z_0|\leq \eta$, so
    $$|\phi(z)-z| \leq \diam \pi_k^{-1}(\{\psi(w),w \}) <\varepsilon.$$
    This completes the proof. 
 \end{proof}

Now we are ready to perturb the regions $V_B$, $B\in \mathcal U$, into regions whose boundaries are contained in $\Omega$. By Lemma \ref{lemma:u} we have 
$\mathcal{C}_L(\Omega) \subset \mathcal{U}$ and hence
$$
\mathcal{U}=\mathcal{C}_L(\Omega) \cup \mathcal{B}, \,\, \textrm{for some}\,\, \mathcal{B} \subset \mathcal{B}'. 
$$ 
Let $J_0= \bigcup_{B \in \mathcal{U}} \partial V_B$. By Lemma \ref{lemma:u'} \ref{lemma:u':jordan}, $J_0$ is the union of finitely many Jordan curves. Moreover, by the last part of Lemma \ref{lemma:u'} \ref{lemma:u':inclusions},
$$
J_0 \cap B = \emptyset \quad \text{for all } B \in \mathcal{U}. 
$$
However, the set $J_0$ might intersect $\hatc\setminus \Omega$. We use Lemma \ref{lemma:moore} to deform slightly $J_0$ into a set $J_1 \subset \Omega$ that consists of Jordan curves bounding quasiround regions. This is made precise in the following lemma. 

\begin{lemma}\label{lemma:qb}
There exists a collection $q_B$, $B\in \mathcal U$, of closed Jordan regions with the following properties. 
\begin{enumerate}[label=\normalfont(\arabic*)]
    \item\label{lemma:qb:boundary} For each $B\in \mathcal U$ we have $\partial q_B\subset \Omega$.
    \item\label{lemma:qb:closure}  For each $B\in \mathcal U$ we have $q_B \subset \hatc\setminus \overline{\Omega_{j-1}}$.
    \item\label{lemma:qb:interiors}  The sets $\operatorname{int}q_B$, $B\in \mathcal U$, are pairwise disjoint.
    \item\label{lemma:qb:cover} $\hatc\setminus \Omega \subset \bigcup_{B\in \mathcal U} q_B$.
    \item\label{lemma:qb:inclusions} For each $B=\overline{\mathbb{D}}(z_B,r_B) \in \mathcal{U}$ we have $${
\overline{\mathbb{D}}(z_B,r_B/2) \subset q_B \subset \overline{\mathbb{D}}(z_B, r_B+5\delta)\subset \overline{\mathbb{D}}(z_B,21r_B) }.$$
    \item\label{lemma:qb:complement} For each $B\in \mathcal U$, if $q_B\setminus \Omega\neq \emptyset$, then $q_B\subset N_{1/j}(\hatc\setminus \Omega)$.
\end{enumerate}
\end{lemma}
\begin{proof}
    For $B\in \mathcal C_L(\Omega)$, by Lemma \ref{lemma:u'} \ref{lemma:u':inclusions} we have $B\subset V_B$. Since $B\subset \hatc\setminus \Omega$, we may find a closed Jordan region $\widetilde B \subset V_B$ such that $\partial \widetilde B\subset \Omega$ and $B\subset \widetilde B$. Recall that all components of $\hatc\setminus \Omega$ that are not in $\mathcal C_L(\Omega)$ have diameter less than $\delta/2$. Hence, all complementary components of the domain $\Omega\cup (\bigcup_{B\in \mathcal C_L(\Omega)}\widetilde B)$ have diameter less than $\delta/2$. It follows that there exist pairwise disjoint closed Jordan regions 
    $$A_1,\dots, A_m\subset \hatc\setminus  \Big(\partial \Omega_{j-1}\cup \Big(\bigcup_{B\in \mathcal C_L(\Omega)}\widetilde B\Big)\Big)$$
    of diameter less than $\delta/2$ such that $\partial A_i\subset \Omega$ for each $i\in \{1,\dots,m\}$, and    
    $$ \hatc\setminus \Omega \subset \Big(\bigcup_{B\in \mathcal C_L(\Omega)} \widetilde B \Big) \cup \Big(\bigcup_{i=1}^m A_i\Big).$$
    See Figure \ref{fig:qb} for an illustration of the regions $\widetilde B$ and $A_i$.
    
\begin{figure}
    \centering
    
\begin{tikzpicture}
\draw (0,0) circle (1cm);
\draw[color=blue,pattern={north east lines},pattern color=blue] (0,0) circle (1.6cm);

\draw[rounded corners=20pt, dashed] (-3,-2)--(3,-2)-- (3,2)--(-3,2)--cycle;

\draw[color=white] (3,2)--(3,-2);

\draw (4.3,1.5) circle (0.3cm);
\draw[color=red,pattern={north east lines},pattern color=red] (4.3,1.5) circle (0.4cm);

\draw (5.5,-1) circle (0.15cm);
\draw (5.4,-0.6) circle (0.2cm);
\draw (5,-0.8) circle (0.2cm);
\draw[color=red,pattern={north east lines},pattern color=red] (5.3,-0.8) circle (0.58cm);

\draw (6.7,0.5) circle (0.1cm);
\draw (7.1,0.6) circle (0.2cm);
\draw[color=red,pattern={north east lines},pattern color=red](6.95,0.5) circle (0.44cm);

\draw[rounded corners=20pt, dashed] (3,-2)--(9,-2)-- (9,2)--(3,2)--cycle;

\draw (3,0) circle (0.3cm);
\draw (2.7,0.5) circle (0.1cm);
\draw (3.1,0.6) circle (0.2cm);
\draw[color=red,pattern={north east lines},pattern color=red] (3,0.26) circle (0.65cm);

\draw (3.07,-1) circle (0.15cm);
\draw[color=red, pattern={north east lines},pattern color=red] (3.07,-1) circle (0.23cm);

\draw (2.4,-0.6) circle (0.2cm);
\draw[color=red,pattern={north east lines},pattern color=red] (2.4,-0.6) circle (0.3cm);

\draw (1,-0.8) circle (0.2cm);

\draw (4,-0.5) circle (0.2cm);
\draw[color=red,pattern={north east lines},pattern color=red] (4,-0.5) circle (0.3cm);

\node[anchor=north west] at (-2.8,1.8) {$V_B$};
\node[anchor=north east] at (8.8,1.8) {$V_{B'}$};
\node at (-1.6,-1) {$\widetilde B$};
\node at (3.3,1.05) {$A_i$};
\end{tikzpicture}

\vspace{1em}

\begin{tikzpicture}
\draw (0,0) circle (1cm);
\draw[color=blue,pattern={north east lines},pattern color=blue] (0,0) circle (1.6cm);

\draw[rounded corners=8pt,dashed] (-3,-2)--(3.5,-2)-- (3.5,-1)-- (2.5,-0.25)--(2.3,0)-- (2.2,0.8)--(3,1.1)--(3,2)--(-3,2)--cycle;

\draw (4.3,1.5) circle (0.3cm);
\draw[color=red,pattern={north east lines},pattern color=red] (4.3,1.5) circle (0.4cm);

\draw (5.5,-1) circle (0.15cm);
\draw (5.4,-0.6) circle (0.2cm);
\draw (5,-0.8) circle (0.2cm);
\draw[color=red,pattern={north east lines},pattern color=red] (5.3,-0.8) circle (0.58cm);

\draw (6.7,0.5) circle (0.1cm);
\draw (7.1,0.6) circle (0.2cm);
\draw[color=red,pattern={north east lines},pattern color=red](6.95,0.5) circle (0.44cm);

\draw[rounded corners=8pt, dashed] (3.5,-1.5)--(3.5,-2)--(9,-2)-- (9,2)--(3,2)--(3,1.5);

\draw (3,0) circle (0.3cm);
\draw (2.7,0.5) circle (0.1cm);
\draw (3.1,0.6) circle (0.2cm);
\draw[color=red,pattern={north east lines},pattern color=red] (3,0.26) circle (0.65cm);

\draw (3.07,-1) circle (0.15cm);
\draw[color=red, pattern={north east lines},pattern color=red] (3.07,-1) circle (0.23cm);

\draw (2.4,-0.6) circle (0.2cm);
\draw[color=red,pattern={north east lines},pattern color=red] (2.4,-0.6) circle (0.3cm);

\draw (1,-0.8) circle (0.2cm);

\draw (4,-0.5) circle (0.2cm);
\draw[color=red,pattern={north east lines},pattern color=red] (4,-0.5) circle (0.3cm);

\node[anchor=north west] at (-2.8,1.8) {$q_B$};
\node[anchor=north east] at (8.8,1.8) {$q_{B'}$};
\node at (-1.6,-1) {$\widetilde B$};
\node at (3.3,1.05) {$A_i$};
\end{tikzpicture}

    \caption{Top: The region $\widetilde B$ (blue) corresponding to $B\in \mathcal C_L(\Omega)$ and some of the regions $A_i$ (red). Bottom: The regions $q_B$ and $q_{B'}$.} 
    \label{fig:qb}
\end{figure}
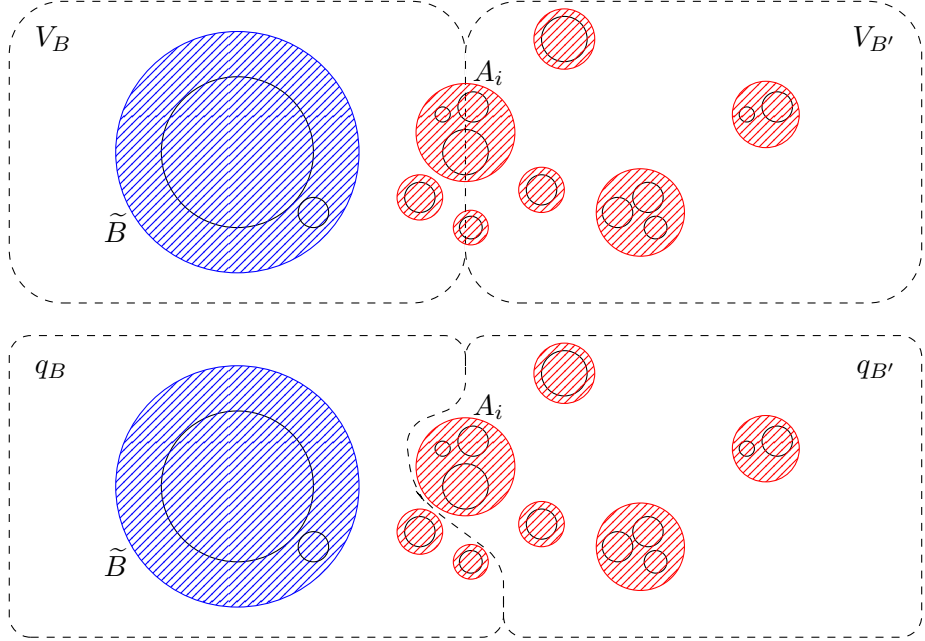
    
    Let $\varepsilon=\delta/2$, $E=\bigcup_{B\in \mathcal U} \partial V_B$, and $F=\bigcup_{B\in \mathcal C_L(\Omega)}\widetilde B$. By the choice of $\widetilde B$, we have $E\cap F=\emptyset$. Moreover, $F\subset \mathbb C\setminus \bigcup_{i=1}^m A_i$. 
    By Lemma \ref{lemma:moore}, we obtain a homeomorphism $\phi\colon \mathbb C\to \mathbb C$ such that $\phi$ is $(\delta/2)$-close to the identity map, and  $\phi(E)$ is disjoint from $\bigcup_{i=1}^m A_i$ and $\bigcup_{B\in \mathcal C_L(\Omega)}\widetilde B$. For $B\in \mathcal U$ define $q_B=\overline{\phi(V_B)}$; see Figure \ref{fig:qb}. We will show that the collection $q_B$, $B\in \mathcal B$, has the desired properties. 
    
    Note that \ref{lemma:qb:boundary} and \ref{lemma:qb:interiors} are immediate by the properties of $\phi$ and the fact that the regions $V_B$, $B\in \mathcal U$, are pairwise disjoint. For $B\in \mathcal U$, by the definition of $\mathcal U$, we have $\overline{V_B}\cap N_{10\delta}(\Omega_{j-1})=\emptyset$. Since $q_B \subset N_{\delta/2}(\overline{V_B})$, we conclude that $q_B \cap \overline{\Omega_{j-1}}=\emptyset$, as required in \ref{lemma:qb:closure}.     For \ref{lemma:qb:cover}, let $p$ be a component of $\hatc\setminus \Omega$. By Lemma \ref{lemma:u}, $N_{10\delta}(p) \subset \bigcup _{B\in \mathcal U}\overline{V_B}$. Using a homotopy argument, based on the fact that $\phi$ is $(\delta/2)$-close to the identity map, one can show that each point of $p$ is surrounded by $\phi(\partial N_{10\delta}(p))$, so $p\subset \bigcup_{B\in \mathcal U} q_B$. 
    
    Let $B=\overline{\mathbb D}(z_B,r_B)\in \mathcal U$. Suppose first that $B\in \mathcal B'$. Since $\phi$ is $(\delta/2)$-close to the identity map and $r_B=\delta$, we conclude (by considering a linear homotopy from $\phi$ to the identity) that 
    $$\frac{1}{2}B \subset \phi(B).$$
    Combining this with Lemma \ref{lemma:u'} \ref{lemma:u':inclusions} we obtain
    $$ \overline{\mathbb D}(z_B, \frac{1}{2}r_B) \subset \phi(B)\subset {\phi(V_B)} \subset q_B \subset N_{\delta/2}(\overline{V_B}) \subset \overline{\mathbb D}(z_B, r_B+ 5\delta).$$
    
    If $B=\overline{\mathbb D}(z_B,r_B)\in \mathcal C_L(\Omega)$, then $r_B\geq \delta/4$ and $B\subset V_B$. As we proved in part \ref{lemma:qb:boundary},  $\partial q_B$ is disjoint from $B$, so we either have $B\subset \inter q_B$ or $B$ is not surrounded by $\partial q_B$. In the latter case, let $\gamma(t,z)= t(\phi(z)-z)+z$, $t\in [0,1]$, and note that a point $z\in \partial V_B$ must satisfy $\gamma(t,z)=z_B$ for some $t\in (0,1)$. Then, the quantity $|\gamma(1,z)-\gamma(0,z)|=|\phi(z)-z|$ is the length of a segment passing through $z_B$ with endpoints outside $B$. Hence, $|\phi(z)-z|\geq \delta/2$, a contradiction. Therefore, $B\subset \inter q_B$. Combining this with Lemma \ref{lemma:u'} \ref{lemma:u':inclusions}, we obtain
    $$ \overline{\mathbb D}(z_B,r_B)\subset q_B\subset N_{\delta/2}(\overline{V_B}) \subset \overline{\mathbb D}(z_B,r_B+5\delta)\subset \overline{\mathbb D}(z_B,21r_B),$$
    given that $\delta\leq 4r_B$. This completes the proof of part \ref{lemma:qb:inclusions}.   

    Finally, we show part \ref{lemma:qb:complement}. Suppose $B\in \mathcal U$ and $q_B\setminus \Omega\neq \emptyset$. By \ref{lemma:qb:inclusions}, we have $q_B \subset N_{6\delta}(\hatc\setminus \Omega)$ when $B\in \mathcal C_L(\Omega)$ and $q_B \subset N_{13\delta}(\hatc \setminus \Omega)$ when $B\in \mathcal B$. The choice of $\delta$ in \eqref{defdelta}  implies that $q_B \subset N_{1/j}(\hatc \setminus \Omega)$.
\end{proof}

We are ready to construct the domain $\Omega_j$. Let 
$\mathcal{V}=\{B \in \mathcal{U}: q_B \setminus \Omega \neq \emptyset\}$. By Lemma \ref{lemma:qb} \ref{lemma:qb:cover}, the set $\hatc \setminus \Omega$ is covered by the collection $\{q_B: \, B \in \mathcal{V}\}$. Also, by \ref{lemma:qb:closure} and \ref{lemma:qb:complement}, each $q_B$ is contained in $\hatc\setminus \overline{\Omega_{j-1}}$ and in $N_{1/j}(\hatc\setminus \Omega)$. We conclude that
\begin{align}\label{quasiroundinclusion2}
   \hatc \setminus \Omega \subset \bigcup_{B \in \mathcal{V}} q_B 
\subset \hatc \setminus \overline{\Omega_{j-1}} \quad \textrm{and} \quad 
\bigcup_{B \in \mathcal{V}} q_B \subset N_{1/j}(\hatc \setminus \Omega). 
\end{align}
Note also that each $q_B$ is $42$-quasiround by \ref{lemma:qb:inclusions} and $\partial q_B\subset \Omega$ by \ref{lemma:qb:boundary}. Thus, if $q_B$, $B\in \mathcal V$, were the complementary components of a domain $\Omega_j$, then \eqref{quasiroundinclusion} would be satisfied and the proof would be completed. Although the sets $q_B$ have disjoint interiors by \ref{lemma:qb:interiors}, their boundaries might intersect, so they are not necessarily the complementary components of a domain.

We hence modify each $q_B$ slightly to amend this. Namely, we consider a closed Jordan region $p_B \subset \inter q_B$ such that $p_B$ is $43$-quasiround, $\partial p_B \subset \Omega$, and \eqref{quasiroundinclusion2} is true with $p_B$ in place of $q_B$.  Then $\Omega_j$ is the domain for which 
$$
\mathcal{C}(\Omega_j) =\{p_B: \, B \in \mathcal{V}\}. 
$$
The proof of Theorem \ref{theorem:quasiround} is complete. \qed

\section{Quasiround exhaustions and limit maps} \label{section:convergence}
Let $\Omega \subset \hatc$ be a circle domain with $\infty \in \Omega$, and $(\Omega_j)_{j \in \N}$ 
be an exhaustion of $\Omega$. Moreover, let $f_j:\Omega_j \to D_j$, $j\in \N$, be the normalized conformal maps from Theorem \ref{mainthm}; that is, each $f_j$ fixes three prescribed points $a_1,a_2,a_3\in \Omega_1$ and each $D_j$ is a finitely connected circle domain. By applying M\"obius transformations if necessary, we may assume that $a_1=\infty=f_j(\infty)$ for every $j \in \mathbb{N}$. Recalling that $(f_j)_{j \in \N}$ has a converging subsequence and that a subsequence of an exhaustion of $\Omega$ is also an exhaustion of $\Omega$, the first claim in Theorem \ref{mainthm} follows from Theorem \ref{theorem:quasiround} and the following result. 

\begin{theorem}\label{theorem:convergence}
    Suppose that $(\Omega_j)_{j \in \N}$ is quasiround and that $(f_j)_{j \in \N}$ converges locally uniformly in $\Omega$ to a conformal homeomorphism $f:\Omega \to D$ for some domain $D\subset \hatc$. Then the sequence of disks $(\hat f_j(p_j(\bar p)))_{j \in \N}$ converges to $\hat f(\bar p)$ in the Hausdorff sense for every $\bar p\in \mathcal C(\Omega)$. In particular, $D$ is a circle domain. 
\end{theorem}
Recall that $p_j(\bar p)$ denotes the unique element of $\mathcal C(\Omega_j)$ that contains $\bar p$. 


\subsection{Outline of the proof of Theorem \ref{theorem:convergence}}\label{section:convergence_outline}

We briefly discuss the main idea of the proof of Theorem \ref{theorem:convergence} given below. Let $\bar p\in\mathcal C(\Omega)$ be as in the theorem. After taking a subsequence, we may assume that $(\hat f_j(p_j(\bar p)))_{j \in \N}$ converges to some disk or point $q(\bar{p})$ in the Hausdorff sense. We need to prove that $q(\bar{p}) = \hat{f}(\bar{p})$. 

By Carath\'eodory's kernel convergence theorem (see \cite{Ntalampekos:uniformization_packing}*{Lemma 2.14}), the inclusion $q(\bar{p}) \subset \hat{f}(\bar{p})$ always holds without any assumptions on the exhaustion $(\Omega_j)_{j \in \mathbb{N}}$. However, the reverse inclusion 
\begin{equation} \label{krungsi}
\hat{f}(\bar{p}) \subset q(\bar{p}) 
\end{equation}
does not always hold; see \cite{Raj23}*{Theorems 1.1 and 1.2}. We will show that the quasiroundness assumption in Theorem \ref{theorem:convergence} implies \eqref{krungsi}. 

Towards a contradiction, we assume that \eqref{krungsi} is not true. 
It follows that there is a point $z_0 \in \partial \bar p$ around which the maps $f_j$, $j\in \N$, do not satisfy some boundary equicontinuity estimates. Specifically, there is a sequence $z_m\in \Omega$, $m\in \N$, converging to $z_0$ such that for each fixed $m$ the sequence $f_j(z_m)$ is uniformly far from $q(\bar p)$ when $j$ is large enough; see Figures \ref{fig:outline} and \ref{fig:fjpj} for an illustration. To reach a contradiction, we thus need to prove equicontinuity estimates for the maps $f_j$ around $z_0$.

To illustrate the technique, for the sake of simplicity, suppose first that $\bar p$ is an isolated complementary component of the domains $\Omega$ and $\Omega_j$ for each $j\in \N$, and that $f_j$ converges locally uniformly in $\Omega$ to a conformal map $f$ on $\Omega$. The desired equicontinuity estimate will be obtained in this case with the aid of \textit{classical modulus}. Classical modulus is a conformally invariant quantity that measures the thickness or richness of a curve family. 

Fix a Jordan curve $J\subset \Omega$ that surrounds $\bar p$, but does not surround any other complementary component of $\Omega$, $\Omega_j$, $j\in \N$. For $m\in\N$ let $\Gamma(m)$ be the family of curves in $\Omega$ that have endpoints in $J$, are contained in the Jordan region bounded by $J$, and separate $\bar p$ from $z_m$. Since $z_m$ converges to $z_0$, the thickness of the family $\Gamma(m)$ is very small and it can be shown that the modulus of $\Gamma(m)$ converges to $0$ as $m\to\infty$; see Figure \ref{fig:outline}. 

The family $f_j(\Gamma(m))$ consists of the curves in $f_j(\Omega)$ that have endpoints on the Jordan curve $f_j(J)$, which converges to the curve $f(J)$, are contained in the Jordan region bounded by $f_j(J)$, and separate $\hat f_j(\bar p)$ from $f_j(z_m)$; see Figure \ref{fig:outline} for an illustration. Given that $f_j(z_m)$ stays uniformly far from the limit $q(\bar p)$ of $\hat f_j(\bar p)$, it can be shown that the curve family $f_j(\Gamma(m))$ has modulus that is uniformly bounded from below as $j\to \infty$ and $m\to\infty$. This lower bound is also related to the so-called \textit{Loewner property} of the plane. By the conformal invariance of modulus, the curve families $\Gamma(m)$ and $f_j(\Gamma(m))$ have the same modulus, which leads to a contradiction.

\begin{figure}
    \centering
    \begin{tikzpicture}
    \begin{scope}[scale=0.5]
    \path[clip, preaction={draw,dashed}][rounded corners=10pt](7,0)--(5,4)--(0,5)--(-4,4)--(-4.5,0)--(-4,-2)--(0,-2)--(4,-2)--cycle;

    \draw (0,0) circle (1cm) node[below] {$\bar p$};
    \fill[black] (0,1) circle (2.5pt) node[below] {$z_0$};

    \node at (4.7,-0.8) {$J$};
    \fill[black] (0,1.8) circle (2.5pt) node[above] {$z_m$};

    \draw[blue, rounded corners=5pt] (-5, 4) -- (0,1.55)-- ( 7,4);
    \draw[blue, rounded corners=5pt] (-5, 3.5) -- (0,1.5)-- ( 7,3.5);
    \draw[blue, rounded corners=5pt] (-5, 3) -- (0,1.4)-- ( 7,3);
    \draw[blue, rounded corners=5pt] (-5, 2.5) -- (0,1.3)-- ( 7,2.5);
    \draw[blue, rounded corners=5pt] (-5, 2) -- (0,1.2)-- ( 7,2);
    \draw[blue, rounded corners=5pt] (-5, 1.5) -- (0,1.15)-- ( 7,1.5);
    \draw[blue, rounded corners=5pt] (-5, 1) -- (0,1.1)-- ( 7,1);
    \draw[blue, rounded corners=5pt] (-5, 0.5) -- (0,1.1)-- ( 7,0.5);
    \end{scope}
    
    \begin{scope}[xshift=6.5cm, scale=0.5]
    \path[clip, preaction={draw,dashed}][rounded corners=10pt](6.5,0)--(6,3)--(5,4)--(2,4.5)--(0,5)--(-4,4)--(-3.8,2)--(-4.5,0)--(-4,-2)--(0,-1.5)--(4,-2)--cycle;

    \draw (0,0) circle (1cm) node {$\hat f_j(\bar p)$};

    \draw[blue] (-5,1.1)--(7,1.1);
    \draw[blue] (-5,1.4)--(7,1.4);
    \draw[blue] (-5,1.7)--(7,1.7);
    \draw[blue] (-5,2)--(7,2);
    \draw[blue] (-5,2.3)--(7,2.3);
    \draw[blue] (-5,2.6)--(7,2.6);
    \draw[blue] (-5,2.9)--(7,2.9);
    \draw[blue] (-5,3.2)--(7,3.2);

    \node at (4.3,-0.8) {$f_j(J)$};
    \fill[black] (0,3.6) circle (2.5pt) node[above] {$f_j(z_m)$};
    \end{scope}

    \draw[->] (3,2) --node[above, pos=0.5]{$f_j$}(4,2);
\end{tikzpicture}
    \caption{The distortion of the curve family $\Gamma(m)$ under $f_j$.}
    \label{fig:outline}
\end{figure}

The main idea of the proof of Theorem \ref{theorem:convergence} is that the above argument can be adapted to conformal maps between quasiround domains even without the condition that $\bar p$ is isolated. However, classical modulus is no longer sufficient to obtain the desired upper and lower modulus estimates. Instead, we rely on \emph{transboundary modulus}, a notion introduced by Schramm in the seminal work \cite{Sch95}; see Section \ref{sec:proofconvergence} for the definition. Transboundary modulus is a conformal invariant that, unlike classical modulus, accounts for the shapes of boundary components. 

We show that finitely connected circle domains and quasiround domains enjoy transboundary modulus estimates that are very similar to estimates satisfied by classical modulus. These estimates are strong enough to prove \eqref{krungsi} following the above strategy; see e.g.\ \cites{Sch95,Bon11,Raj23} for similar estimates. More precisely, we prove the estimate \eqref{annulusestimate}, which gives an upper bound for the transboundary modulus of (a variant of) the curve family $\Gamma(m)$ discussed above. On the other hand, we obtain a lower bound for the transboundary modulus of the image of this curve family under $f_j$ (see Lemma \ref{lowerboundcircle} and Figure \ref{fig:outline}); this estimate is related to the {Loewner property} of the plane with respect to transboundary modulus. Finally, as in the simplified case, these two estimates, combined with the conformal invariance of transboundary modulus, yield \eqref{krungsi}.

\subsection{Proof of Theorem \ref{theorem:convergence}} \label{sec:proofconvergence}
We recall the definition of the transboundary modulus, as introduced by Schramm \cite{Sch95}. Let $G\subset \hatc$ be a domain. Let $\rho\colon \hat G \to [0,\infty]$ be a Borel function and $\gamma\colon [a,b]\to  \hat G$ be a curve. Then $\gamma^{-1}( \pi_{G}(G))$ has countably many components $O_j\subset [a,b]$, $j\in J$. For $j\in J$ define $\gamma_j= \gamma|_{O_j}$ and $\alpha_j = \pi_{G}^{-1}\circ \gamma_j$. We define
\begin{align*}
\int_{\gamma} \rho \, ds= \sum_{j\in J}\int_{\alpha_j} \rho\circ \pi_{G} \, ds,
\end{align*}
where the integral is understood to be infinite if one of the curves $\alpha_j$ is not locally rectifiable. Let $\Gamma$ be a family of curves in $\hat \Omega $.  We say that a Borel function $\rho\colon \hat G \to [0,\infty]$ is \textit{admissible} for $\Gamma$ if 
\begin{align*}
\int_{\gamma} \rho \, ds + \sum_{\substack{p\in \mathcal C(G)\\ |\gamma|\cap p\neq \emptyset}} \rho( p)\geq 1
\end{align*}
for each $\gamma\in \Gamma$. Here $|\gamma|$ denotes the image of $\gamma$. The \textit{transboundary modulus} of $\Gamma$ with respect to the domain $G$ is defined to be
\begin{align*}
\modu_{G}\Gamma=  \inf_{\rho}\left\{ \int_{G}(\rho\circ \pi_{G})^2 \, dA + \sum_{p\in \mathcal C(G)}\rho(p)^2\right\},
\end{align*}
where the infimum is taken over all admissible functions $\rho$.  It was observed by Schramm that transboundary modulus is invariant under conformal maps.  Specifically, if $f\colon G\to G'$ is a conformal map between domains $G,G'\subset \hatc$, then for every curve family $\Gamma$ in $\hat G$ we have 
$\modu_{G}\Gamma  =\modu_{G'} \hat f(\Gamma)$.

The rest of this section is devoted to the proof of Theorem \ref{theorem:convergence}. Fix $\bar{p} \in \mathcal{C}(\Omega)$ and let $J \subset \Omega_1$ be a Jordan curve that separates $p_1(\bar{p})$ and $\infty$. Denote the bounded component of $\hatc \setminus J$ by $U$. Recall that $f_j$ extends to a homeomorphism $\hat{f}_j:\hat{\Omega}_j \to \hat{D}_j$. Consider a compact set $q(\bar{p}) \subset \mathbb{C}$ that is the Hausdorff limit of a subsequence of $(\hat{f}_j(p_j(\bar{p})))_{j \in \N}$. Then $q(\bar{p})$ is a disk or a point, and $q(\bar{p}) \subset \hat{f}(\bar{p})$, as a consequence of Carath\'eodory's kernel convergence theorem; see \cite{Ntalampekos:uniformization_packing}*{Lemma 2.14}. Theorem \ref{mainthm} follows if we can show that $q(\bar{p})=\hat{f}(\bar{p})$.

Towards a contradiction, suppose that $f(\bar{p}) \setminus q(\bar{p}) \neq \emptyset$. Then there is a number $\delta>0$ and a sequence of points $(z_m)_{m \in \N}$ so that $z_m \in \partial p_m(\bar{p})$ and 
\begin{equation} \label{bycontra}
\liminf_{j \to \infty} \dist(f_j(z_m),\hat{f}_j(p_j(\bar{p}))) \geq 2\delta \quad \text{for every } m=1,2,\ldots. 
\end{equation}
See Figure \ref{fig:fjpj} for an illustration. Passing to a subsequence if necessary, we may assume that $z_m \to z_0 
\in \partial \bar{p}$ and that $W_m \subset U$ for every $m \geq 1$, where 
$$
W_m=\overline{\mathbb{D}}(z_0,|z_m-z_0|). 
$$
For $m\in \N$ fix $j(m) >m$ so that $z_m \in \Omega_j$ and 
\begin{equation}\label{deltaestim}
\dist(f_j(z_m),\hat{f}_j(p_j(\bar{p})))>\delta \quad \text{for every } j \geq j(m),
\end{equation}
and let $\Gamma(j,m)$ be the family of curves in $\hat{\Omega}_j$ joining $J$ and $\pi_{\Omega_j}(W_m)$ 
without intersecting $p_j(\bar{p})$. We need the following estimate on the circle domains $D_j$. 

\begin{lemma}
\label{lowerboundcircle}
There exists $N>0$ such that
\begin{equation} \label{kauno}
\modu_{D_j} \hat{f}_j(\Gamma(j,m)) \geq N>0 
\quad \text{for every $j \geq j(m)$ and $m\geq 1$}. 
\end{equation}
\end{lemma}

\begin{remark}
The estimate in the conclusion Lemma \ref{lowerboundcircle} follows from the fact that circle domains have the \textit{Loewner property with respect to transboundary modulus}. This was first observed by Bonk \cite{Bon11}*{Proposition 8.1} under the additional assumption that the complementary disks of the domain are \textit{uniformly relatively separated}. Then Hakobyan and Li \cite{HakLi19}*{Theorem 4.3} proved it for all circle domains. See also \cite{Nta25:schottky}*{Proposition 5.7} for an even more general version of this property. Nevertheless, for the sake of completeness and for the convenience of the reader we include the proof of Lemma \ref{lowerboundcircle}.
\end{remark}

\begin{proof}
Let $m \geq 1$ and $j \geq j(m)$, and let $y_m$ be the point on the disk $\hat{f}_j(p_j(\bar p))$ that is closest to $f_j(z_m)$. Given $0<t<\delta$, let $\ell_t \subset \mathbb{C}$ be the line that is orthogonal to the line segment $(y_m,f_j(z_m))$ and passes through the point $z_t\in (y_m,f_j(z_m))$ that satisfies $|z_t-y_m|=t$; by \eqref{deltaestim} such a point exists.

The set $\hat{f}_j(\pi_{\Omega_j}(W_m))$ contains a continuum that joins $f_j(z_m)$ and $\hat{f}_j(p_j(\bar p))$, and is disjoint from $f_j(J)$. It follows that $\pi_{D_j}(\ell_t)$ intersects $\hat{f}_j(\pi_{\Omega_j}(W_m))$ for every $0<t<\delta$. If we denote the bounded component of $\hatc \setminus f_j(J)$ by $V_j$, we see that every such $\ell_t$ contains a (parametrized) line segment $I_t \subset \overline{V_j}$ so that $\gamma_t=\pi_{D_j} \circ I_t$ connects $f_j(J)$ and $\hat{f}_j(\pi_{\Omega_j}(W_m))$ and thus belongs to $\hat{f}_j(\Gamma(j,m))$; see Figure \ref{fig:fjpj}.  

We will now apply a variant of the classical length-area method with the curves $\gamma_t$ to prove \eqref{kauno}. Let $\rho$ be an admissible function for $\hat{f}_j(\Gamma(j,m))$. Since $\gamma_t \in \hat{f}_j(\Gamma(j,m))$, we have 
$$
1 \leq \int_{\gamma_t} \rho \, ds + \sum_{\substack{p\in \mathcal C(D_j)\\ |\gamma_t|\cap p\neq \emptyset}} \rho( p)\eqqcolon I(t)+S(t) \quad \text{for every $0<t<\delta$}. 
$$
Integrating both sides over $0<t<\delta$ yields 
\begin{equation}\label{eq:warm1} 
\delta \leq \int_0^\delta I(t) \, dt + \int_0^\delta S(t) \, dt. 
\end{equation} 

By the locally uniform convergence of $f_j$ to $f$, there exists $L>0$ so that $\overline{V_j} \subset \mathbb{D}(0,L)$ for every $j \geq 1$. Recall also that $I_t \subset \overline{V_j}$ for every $0<t<\delta$. Thus, by Fubini's theorem and H\"older's inequality we have 
\begin{equation} 
\label{eq:warm2}
\begin{split}   
\int_0^\delta I(t) \, dt &\leq \int_{\overline{V_j} \cap D_j} \rho \circ \pi_{D_j} \, dA 
\leq \area(\overline{V_j})^{1/2} \Big(\int_{D_j} (\rho \circ \pi_{D_j})^2 \, dA \Big)^{1/2} \\
&\leq \pi^{1/2} L \Big(\int_{D_j} (\rho \circ \pi_{D_j})^2 \, dA \Big)^{1/2}. 
\end{split}
\end{equation}

Next, given $p \in \mathcal{C}(D_j)$ we define 
$$
I_p=\{0<t<\delta: \, |\gamma_t| \cap p \neq \emptyset\} \subset (0,\delta), 
$$ 
and denote the length of $I_p$ by $\ell(I_p)$. We have 
\begin{eqnarray*}
\int_0^\delta S(t)\, dt= \sum_{p \in \mathcal{C}(D_j)} \ell(I_p)\rho(p) 
\leq \Big(\sum_{p \in \mathcal{C}(D_j)} \ell(I_p)^2\Big)^{1/2} 
\Big(\sum_{p \in \mathcal{C}(D_j)} \rho(p)^2\Big)^{1/2}. 
\end{eqnarray*}
Since the sets $p$ are disks, we have 
$$
\ell(I_p)^2 \leq \diam(p)^2 = \frac{4 \area(p)}{\pi}.  
$$ 
Moreover, if $\ell(I_p)>0$ then $p \subset V_j \subset \mathbb{D}(0,L)$. Since the disks $p\in \mathcal C(D_j)$ are pairwise disjoint, we conclude from the above estimates that 
\begin{equation}
\label{eq:warm3} 
\begin{split} 
\int_0^\delta S(t)\, dt &\leq \frac{2}{\pi^{1/2}}\Big(\sum_{\substack{p \in \mathcal{C}(D_j)\\ p \subset V_j }} \area(p)\Big)^{1/2} 
\Big(\sum_{p \in \mathcal{C}(D_j)} \rho(p)^2\Big)^{1/2} \\
&\leq \frac{2 \area(V_j)^{1/2}}{\pi^{1/2}} \Big(\sum_{p \in \mathcal{C}(D_j)} \rho(p)^2\Big)^{1/2}  
\leq 2L \Big(\sum_{p \in \mathcal{C}(D_j)} \rho(p)^2\Big)^{1/2}. 
\end{split}
\end{equation}
Combining \eqref{eq:warm1}, \eqref{eq:warm2}, and \eqref{eq:warm3}, we have 
\begin{equation} \label{eq:warm4} 
\begin{split}
\delta^2 &\leq  \Big(\pi^{1/2}L\Big(\int_{D_j} (\rho \circ \pi_{D_j})^2 \, dA \Big)^{1/2} +2L\Big(\sum_{p \in \mathcal{C}(D_j)} \rho(p)^2\Big)^{1/2}\Big)^2 \\
&\leq 8L^2 \Big(\int_{D_j} (\rho \circ \pi_{D_j})^2 \, dA+\sum_{p \in \mathcal{C}(D_j)} \rho(p)^2\Big).  
\end{split}
\end{equation} 
Since \eqref{eq:warm4} holds for all admissible functions and the numbers 
$\delta$ and $L$ are independent of $j$ and $m$, \eqref{kauno} follows from the definition of the transboundary modulus.  
\end{proof}

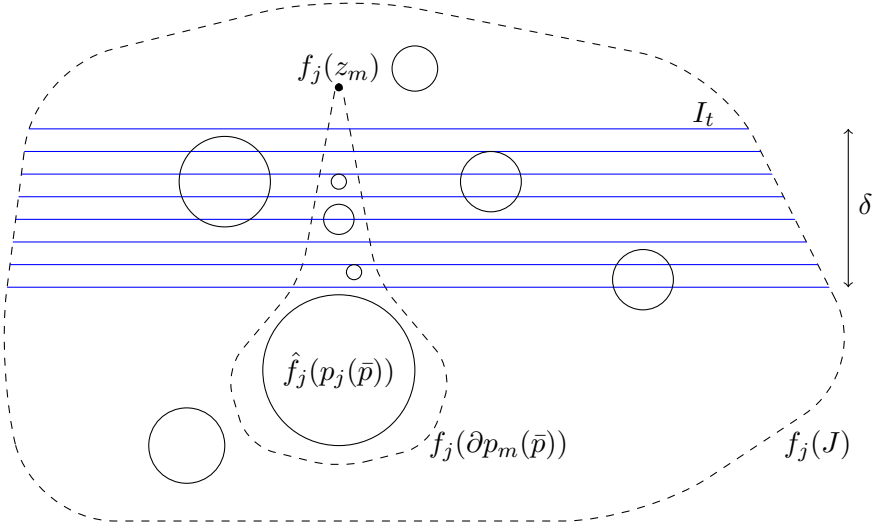
\begin{figure}
    \centering
    \begin{tikzpicture}
\begin{scope}
\path[clip, preaction={draw,dashed}][rounded corners=30pt](7,0)--(5,4)--(0,5)--(-4,4)--(-4.5,0)--(-4,-2)--(0,-2)--(4,-2)--cycle;

\draw (0,0) circle (1cm);

\draw[dashed, rounded corners=10pt] (0,-1.3)--(1.2,-1)--(1.5,0)--(0.5,1.2)--(0, 4)--(-0.5,1.2)--(-1.5,0)--(-1.2,-1)--cycle;

\draw[blue] (-5,1.1)--(7,1.1);
\draw[blue] (-5,1.4)--(7,1.4);
\draw[blue] (-5,1.7)--(7,1.7);
\draw[blue] (-5,2)--(7,2);
\draw[blue] (-5,2.3)--(7,2.3);
\draw[blue] (-5,2.6)--(7,2.6);
\draw[blue] (-5,2.9)--(7,2.9);
\draw[blue] (-5,3.2)--(7,3.2);

\draw (0,2) circle (0.2cm);
\draw (1,4) circle (0.3cm);
\draw (0,2.5) circle (0.1cm);
\draw (0.2,1.3) circle (0.1cm);

\draw (2,2.5) circle (0.4cm);
\draw (4,1.2) circle (0.4cm);
\draw (-2,-1) circle (0.5cm);
\draw (-1.5,2.5) circle (0.6cm);
\end{scope}

\node at (0,0) {$\hat f_j(p_j(\bar p))$};
\node at (2.1,-1) {$f_j(\partial p_m(\bar p))$};
\node at (6.3,-1) {$f_j(J)$};
\fill[black] (0,3.75) circle (1.5pt);
\node at (0,4) {$f_j(z_m)$};
\node at (4.8,3.4) {$I_t$};

\draw[<->] (6.7,1.1)--(6.7,3.2);
\node[anchor=west] at (6.7,2.2) {$\delta$};
\end{tikzpicture}
    \caption{Illustration of \eqref{deltaestim} and of the proof of Lemma \ref{lowerboundcircle}.}
    \label{fig:fjpj}
\end{figure}

In view of Lemma \ref{lowerboundcircle} and the conformal invariance of transboundary modulus, a contradiction to \eqref{bycontra} follows if we can show that
\begin{equation}
\label{annulusestimate}
\lim_{m \to \infty} \lim_{j \to \infty} \modu_{\Omega_j} \Gamma(j,m) =0. 
\end{equation}

To prove \eqref{annulusestimate}, we may assume that $z_0=0$ and $\dist(0,J)=1$. We fix a sequence 
$$
\mathbb{A}_k=\mathbb{D}(0,R_k) \setminus \overline{\mathbb{D}}(0,R_k/2), \quad k=1,2,\ldots, 
$$  
of annuli, where radii $R_k$ are chosen as follows: $R_1=1$, and if $R_{k-1}$ is defined then $R_k$ is the largest number so that $R_k\leq R_{k-1}/4$ and 
so that no $p \in \mathcal{C}(\Omega)$ other than $\bar{p}$ intersects both $\mathbb{S}(0,R_{k-1}/2)$ and $\mathbb{D}(0,2R_k)$. 

\begin{lemma}\label{upperboundquasiround}
There exists $M>0$ such that
\begin{equation} \label{uniformannulus}
\limsup_{j \to \infty} \modu_{\Omega_j} \Lambda(j,k) \leq M \quad \text{for every } k=1,2,\ldots, 
\end{equation} 
where $\Lambda(j,k)$ is the family of curves joining  
$$ 
\pi_{\Omega_j}(\mathbb{S}(0,R_k)) \quad \text{and} \quad  
\pi_{\Omega_j}(\mathbb{S}(0,R_k/2)) \quad \text{in} \quad \hat{\Omega}_j \setminus \{\pi_{\Omega_j}(p_j(\bar{p}))\}. 
$$
\end{lemma}
\begin{proof}
Fix $k\geq 1$ and notice that if $\mathcal B$ is a family of pairwise disjoint disks of radius larger than $R_k/10$ that intersect $\overline{\mathbb D}(0,R_k)$, then $\# \mathcal B\leq 400$. Thus, at most $400$ elements $p\in \mathcal C(\Omega \cup \bar p)$ can have radius larger than $R_k/10$ and intersect $\overline{\mathbb A_k}$. We conclude that there exists $j_0=j_0(k)$ so that if $j \geq j_0$ then there are at most $400$ elements $p \in \mathcal{C}(\Omega_j \cup p_j(\bar{p}))$ 
with diameter greater than $R_k/4$ intersecting $\mathbb{A}_k$. We denote the collection of such elements by $P_L=P_L(j)$. Also, let 
$P_S=P_S(j)$ be the elements of $\mathcal{C}(\Omega_j\cup p_j(\bar{p})) \setminus P_L$ that intersect $\mathbb A_k$. We define
\begin{eqnarray*}
\rho(p)=\left\{  \begin{array}{ll}
1, & p \in P_L, \\
2R_k^{-1} \diam p, & p \in P_S, \\ 
2R^{-1}_k, & p \in \Omega_j\cap \mathbb A_k.  
\end{array}
\right. 
\end{eqnarray*} 
Then $\rho$ is admissible for $\Lambda(j,k)$, and 
$$
\int_{\Omega_j} \rho^2 \, dA +\sum_{p \in P_L}\rho(p)^2 \leq 4\pi +400. 
$$
To estimate the sum of values $\rho(p)^2$ over $P_S$, we recall that $\Omega_j$ is $K$-quasiround for some $K\geq 1$ by assumption. In particular, 
$$
(\diam p)^2 \leq \frac{4K^2 \area(p)}{\pi} \quad \text{for every } p \in P_S, 
$$ 
so 
\begin{equation} \label{PSupper}
\sum_{p \in P_S}\rho(p)^2 \leq  \sum_{p \in P_S}\frac{16K^2\area(p)}{\pi R_k^{2}}\leq 
\frac{16K^2 \area(\mathbb{D}(0,2R_k))}{\pi R_k^2} \leq 64K^2. 
\end{equation} 
Combining the estimates and letting $j \to \infty$ gives \eqref{uniformannulus}.     
\end{proof}

We remark that \eqref{PSupper} is the only estimate in the proof of Theorem \ref{theorem:convergence} which depends on the quasiroundness assumption. 

We are now ready to prove \eqref{annulusestimate}. By Lemma \ref{upperboundquasiround}, given $\ell \in \mathbb{N}$ there is $j'(\ell)$ so that if $1 \leq k \leq \ell$ and $j \geq j'(\ell)$, 
then $\modu_{\Omega_j} \Lambda(j,k) \leq 2M$ and no $p \in \mathcal{C}(\Omega_j)$ other than $p_j(\bar{p})$ intersects both $\mathbb{A}_k$ and $\mathbb{A}_{k+1}$. 
Let $\rho_{j,k}$ be an admissible function for $\Lambda(j,k)$ that satisfies 
$$
\int_{\Omega_j} \rho_{j,k}^2 \, dA + \sum_{p \in \mathcal{C}(\Omega_j)} \rho_{j,k}(p)^2 \leq 3M. 
$$
Now, since $\diam W_m \to 0$ as $m \to \infty$, there is $m(\ell)$ so that if $m \geq m(\ell)$ then $\rho=\ell^{-1}\sum_{k=1}^\ell \rho_{j,k}$ is admissible for $\Gamma(j,m)$ for all $j\geq \max\{j'(\ell),j(m)\}$. Therefore,
$$
\modu_{\Omega_j} \Gamma(j,m)\leq \int_{\Omega_j} \rho^2 \, dA + \sum_{p \in \mathcal{C}(\Omega_j)} \rho(p)^2 \leq 3M\ell^{-1} 
$$
for all $j \geq \max\{j'(\ell),j(m)\}$. Now, \eqref{annulusestimate} follows by letting first $j\to\infty$, then $m\to\infty$, and finally $\ell \to \infty$. The proof of Theorem \ref{theorem:convergence} is complete. \qed

\section{Definition of quasiconformality}\label{section:definition_qc}
We have proved the first claim in Theorem \ref{mainthm} by constructing an exhaustion of a given circle domain $\Omega \subset \hatc$ so that the image of the associated limit map $f\colon \Omega \to D$ is a circle domain. We now start proving the second claim in Theorem \ref{mainthm}: 
if $\partial \Omega$ has area zero, then $D=\Omega$ and $f$ is the identity map. Our strategy is to use a variant of a recent characterization of quasiconformality due to the first author \cite{Ntalampekos:metric_definition_qc} to show that $f$ extends to a conformal homeomorphism of $\hatc$. 

Let $A\subset \mathbb C$ be a bounded open set. The \textit{eccentricity} $E(A)$ of $A$ is the infimum of all numbers $H\geq 1$ for which there exists an open ball $B$ such that $B\subset A\subset HB$. Let $g\colon U\to V$ be a homeomorphism between open sets $U,V\subset \mathbb C$. The \textit{eccentric distortion of $g$} at a point $x\in U$, denoted by $E_g(x)$, is the infimum of all values $H\geq 1$ such that there exists a sequence of open sets $A_n\subset U$, $n\in \mathbb N$, containing $x$ with $\diam A_n\to 0$ as $n\to\infty$ and with the property that $E(A_n)\leq H$ and $E(g(A_n))\leq H$ for each $n\in \mathbb N$. 

\begin{theorem}\label{theorem:eccentric}
Let $g\colon U\to V$ be an orientation-preserving homeomorphism between open sets $U,V\subset \mathbb C$. Let $G\subset U$ be a set with the property that 
$$\mathcal H^1(g(|\gamma|\cap G))=0$$
for a.e.\ horizontal and a.e.\ vertical line segment $\gamma$ in $U$. Suppose that there exists $H\geq 1$ such that $E_g(x)\leq H$ for each point $x\in U\setminus G$. Then $g$ is quasiconformal in $U$, quantitatively.  Moreover, if $G$ is measurable and $H=1$, then $g$ is conformal in $U$.
\end{theorem}

Recall that $|\gamma|$ denotes the image of the path $\gamma$. Also, see \cite{Heinonen:metric}*{\S 8.3} for the definitions of Hausdorff $1$-measure $\mathcal H^1$ and Hausdorff $1$-content $\mathcal H^1_{\infty}$. We recall the definition of the $2$-modulus. Let $\Gamma$ be a family of curves in $\mathbb C$. A Borel function $\rho\colon \mathbb C\to [0,\infty]$ is \textit{admissible} for $\Gamma$ if 
$$ \int_{\gamma}\rho\, ds\geq 1$$
for all locally rectifiable curves $\gamma\in \Gamma$. The $2$-modulus of $\Gamma$ is defined to be
$$ \modu\Gamma = \inf_{\rho} \int \rho^2 \, dA,$$
where the infimum is taken over all admissible functions. The proof of Theorem \ref{theorem:eccentric} is a slight modification of the proof of \cite{Ntalampekos:metric_definition_qc}*{Theorem 3.3}. 

\begin{remark}
    The definition of quasiconformality using the eccentric distortion in Theorem \ref{theorem:eccentric} is motivated by the well-known metric definition of quasiconformality due to Gehring \cite{Gehring:Rings}. The metric definition confines the distortion of infinitesimal \textit{centered balls} under a quasiconformal map $g$. Instead, the definition that uses the eccentric distortion allows for more flexibility, confining the behavior of $g$ on a sequence of \textit{arbitrary open sets} (with bounded eccentricity) shrinking to each point.
\end{remark}

\begin{proof}
    By the assumption regarding $E_g$ and \cite{Ntalampekos:metric_definition_qc}*{Theorem 3.1}, there exists a curve family $\Gamma_0$ with $\modu\Gamma_0=0$ and a Borel function $\rho_g\colon U\to [0,\infty]$ with $\rho_g\in L^2_{\loc}(U)$ such that for all curves $\gamma\notin \Gamma_0$ contained in $U$ we have
    \begin{align}\label{theorem:eccentric:upper_gradient}
        \mathcal H^1_{\infty}( g(|\gamma|\setminus G))\leq \int_{\gamma}\rho_g\, ds
    \end{align}
    and for each Borel function $\rho\colon V\to [0,\infty]$ we have
    \begin{align}\label{theorem:eccentric:area}
        \int_U (\rho\circ g)\cdot \rho_g^2 \, dA \leq c(H) \int_V \rho\, dA.
    \end{align}

    We will show that for each open rectangle $Q$ with $Q\subset \overline{Q}\subset U$ and sides parallel to the coordinate axes we have
    \begin{align}\label{theorem:eccentric:gehring:vaisala}
        \modu \Gamma(Q) \leq c(H)\modu g(\Gamma(Q)),
    \end{align}
    where $\Gamma(Q)$ denotes the family of curves joining the horizontal (resp.\ vertical) sides of $Q$. By a result of Gehring--V\"ais\"al\"a \cite{GehringVaisala:qc_geometric}*{Theorem 2}, this implies that $g$ is quasiconformal, quantitatively. Let $Q\subset\overline{Q}\subset U$ be a rectangle with sides parallel to the coordinate axes and $\rho\colon V\to [0,\infty]$ be a Borel function that is admissible  for $g(\Gamma(Q))$.  
    
    Without loss of generality, suppose that $\Gamma(Q)$ is the family of curves joining the left and right sides of $Q$. For a.e.\ horizontal line segment $\gamma\in \Gamma(Q)$, $\gamma\colon [a,b]\to Q$, and for $[s,t]\subset [a,b]$ we have
    \begin{align*}
        |g(\gamma(t))-g(\gamma(s))|&\leq \mathcal H^1_{\infty}(g(\gamma([s,t]))) = \mathcal H^1_{\infty}(g(\gamma([s,t])\setminus G)) \leq \int_{\gamma|_{[s,t]}}{\rho_g}\, ds<\infty,
    \end{align*}
    where use used the assumption on the set $G$, \eqref{theorem:eccentric:upper_gradient}, and Fubini's theorem (to guarantee the finiteness of the last integral). This implies that (e.g., see \cite{Vaisala:quasiconformal}*{Theorem 5.3})
    $$ \int_{\gamma} (\rho\circ g)\cdot \rho_g\, ds\geq \int_{g\circ \gamma}\rho\, ds\geq 1.$$
    Let $\Gamma'(Q)$ be the family of horizontal line segments $\gamma\in \Gamma(Q)$ satisfying the above. Since $\Gamma'(Q)$ contains a.e.\ horizontal line segment in $\Gamma(Q)$, it is straightforward to show that $\Gamma(Q)$ and $\Gamma'(Q)$ have the same $2$-modulus. Also, $(\rho\circ g)\cdot \rho_g$ is admissible $\Gamma'(Q)$. Hence by \eqref{theorem:eccentric:area}, 
    \begin{align*}
        \modu\Gamma(Q)=\modu\Gamma'(Q) \leq \int_U (\rho\circ g)^2\cdot \rho_g^2\, dA \leq c(H) \int_{V} \rho^2 \, dA. 
    \end{align*}
    This implies the desired \eqref{theorem:eccentric:gehring:vaisala} and proves the first part of the theorem.

    If the set $G$ is measurable then $g(G)$ is also measurable by quasiconformality. We set $f=g^{-1}=u+iv$, which is quasiconformal, and we have
    $$\int_{g(G)} |\nabla v| \, dA =\int \mathcal H^1(v^{-1}(t)\cap g(G)) \, dt$$
    by the coarea formula for Sobolev functions \cite{MalySwansonZiemer:coarea}*{Theorem 1.1}. By assumption, for a.e.\ $t\in \mathbb R$ we have $\mathcal H^1(v^{-1}(t)\cap g(G))=0$.  Hence $|\nabla v|\chi_{g(G)}=0$ a.e. By quasiconformality, we cannot have $|\nabla v|=0$ on a set of positive measure, since this would imply that $J_f=0$ on a set of positive measure. Hence, $\area (g(G))=0$. Again, by quasiconformality, $\area(G)=0$. Therefore, if $H=1$, then $E_g(x)=1$ for a.e.\ $x\in U$. This now implies that $f$ is conformal, as shown in \cite{Ntalampekos:rigidity_cned}*{Lemma 2.5}.
\end{proof}

\section{Regularity of limiting map}\label{section:regularity}
As in the first conclusion of Theorem \ref{mainthm}, let $\Omega\subset \hatc$ be a circle domain  and let $(\Omega_j)_{j \in \N}$ be a quasiround exhaustion of $\Omega$ such that $(f_j)_{j \in \N}$ converges locally uniformly to a conformal homeomorphism $f$ from $\Omega$ onto a circle domain $D\subset \hatc $. In addition suppose that $\partial \Omega$ has $2$-measure zero. Let $g=f^{-1}$ and $g_j=f_j^{-1}$, $j\in \mathbb N$.
We consider the derivative $|Dg|\colon D\to (0,\infty)$ in the spherical metric of $\hatc $. If $z,g(z)\in \mathbb C$, then
\begin{align*}
|Dg|(z)= \frac{1+|z|^2}{1+|g(z)|^2}|g'(z)|.
\end{align*}
Throughout the section, we use the spherical metric $\sigma$ and measure $\Sigma$ on $\hatc$, even if this is not explicitly stated. In particular, line integrals and $2$-modulus are computed with respect to the spherical metric.  Our main goal in this section is to show the following statement, under the above assumptions.

\begin{proposition}[Transboundary upper gradient inequality]\label{proposition:upper_gradient}
    There exists a family of curves $\Gamma_0$ in $\hatc$ with $\modu\Gamma_0=0$ such that for all curves $\gamma\colon [a,b]\to \hatc$ outside $\Gamma_0$ with $\gamma(a),\gamma(b)\in D$ we have
    \begin{align*}
        \sigma(g(\gamma(a)),g(\gamma(b)))\leq \int_{\gamma} |Dg|\chi_D\, ds+ \sum_{\substack{q\in \mathcal C(D) \\q\cap |\gamma|\neq \emptyset}}\diam \hat g(q).
    \end{align*}
\end{proposition}

This statement implies that the conformal map $g$ has some additional regularity, beyond the interior of the domain $D$. Namely, $g$ is in some sense absolutely continuous not only on curves in the domain but also on curves that travel through complementary components. We will need several preparatory statements. In the next statements, closed disks can be degenerate, i.e., they can have radius zero.

\begin{lemma}[{\cite{Ntalampekos:uniformization_packing}*{Lemma 4.14}}]\label{lemma:sum_limsup}
    For each $n\in \mathbb N$, let $q_{i,n}$, $i\in I\cap \{1,\dots,n\}$, where $I\subset \N$, be a collection of pairwise disjoint closed disks in $\hatc$. Suppose that there exists a collection of pairwise disjoint closed disks $q_i$, $i\in I$, with the property that 
    $$\lim_{n\to\infty} q_{i,n}=q_i$$
    for each $i\in I$, in the Hausdorff sense. Then for each non-negative sequence $(\lambda_i)_{i\in I} \in \ell^2(I)$ there exists a family of curves $\Gamma_0$ in $\hatc$ with $\modu\Gamma_0=0$ such that for all curves $\gamma\notin \Gamma_0$ we have
    \begin{align*}
        \limsup_{n\to\infty} \sum_{i:q_{i,n}\cap |\gamma|\neq \emptyset}\lambda_i\leq \sum_{i:q_i\cap |\gamma|\neq \emptyset} \lambda_i.
    \end{align*}
\end{lemma}

\begin{lemma}\label{lemma:sum_tail}
    For each $n\in \N$, let $q_{i,n}$, $i\in I_n$, be a collection of pairwise disjoint closed disks on $\hatc$ and $(\lambda_{i,n})_{i\in I_n}$ be a non-negative sequence with 
    $$\lim_{n\to\infty}\sum_{i\in I_n}\lambda_{i,n}^2 =0 \quad \big(\textrm{resp. } \sum_{i\in I_n}\lambda_{i,n}^2<\infty \textrm{ for each $n\in \N$}\big).$$
    Then there exists a family of curves $\Gamma_0$ in $\hatc$ with $\modu\Gamma_0=0$ such that for all curves $\gamma\notin \Gamma_0$ we have
    $$\lim_{n\to\infty}\sum_{i: q_{i,n}\cap |\gamma|\neq \emptyset}\lambda_{i,n}=0 \quad \big(\textrm{resp. } \sum_{i: q_{i,n}\cap |\gamma|\neq \emptyset}\lambda_{i,n} <\infty \textrm{ for each $n\in \N$}\big).$$
\end{lemma}

\begin{proof} 
    Note that the conclusion is true for constant curves as 
    \begin{align}\label{lemma:sum_tail:sup}
        \lim_{n\to\infty}\sup_{i\in I_n}\lambda_{i,n}=0 \quad \big(\textrm{resp. } \sup_{i\in I_n}\lambda_{i,n}<\infty.\big)
    \end{align}
    Let $J_n\subset I_n$ be the set of indices $i\in I_n$ such that $\diam q_{i,n}>0$. For $n\in \N$ define 
    $$\phi_n= \sum_{i\in J_n} \frac{\lambda_{i,n}}{\diam q_{i,n}} \chi_{2q_{i,n}},$$
    where $2q_{i,n}$ denotes the disk with the same center as $q_{i,n}$ and twice the radius (in the spherical metric). By a variation of Bojarski's lemma (\cite{Boj88}, \cite{Ntalampekos:uniformization_packing}*{Lemma 2.7}), we have
    $$ \| \phi_n\|_{L^2(\hatc)}^2 \leq C \sum_{i\in J_n} \lambda_{i,n}^2.$$
    By assumption, we have $\phi_n\to 0$ in $L^2(\hatc)$ (resp.\ $\phi_n\in L^2(\hatc)$). Hence, by Fuglede's lemma  \cite{HeinonenKoskelaShanmugalingamTyson:Sobolev}*{p.\ 131}, there exists a curve family $\Gamma_1$ with $\modu\Gamma_1=0$ such that 
    \begin{align}\label{lemma:sum_tail:integral}
        \lim_{n\to\infty }\int_{\gamma}\phi_n\, ds=0 \quad \big(\textrm{resp.\ } \int_{\gamma}\phi_n\, ds<\infty\big)
    \end{align}
    for $\gamma\notin \Gamma_1$. Let $\gamma\notin \Gamma_1$ be a non-constant curve. Observe that there exists $N\in \N$, depending on $\gamma$, such that the set
    $$K_n= \{ i\in J_n: \diam (2q_{i,n}) \geq \diam (|\gamma|) >0\}$$
    has at most $N$ elements; to see this, compare the area of $\bigcup_{i\in K_n} q_{i,n}$ with the area of $\hatc$. Thus, for $i\in J_n\setminus K_n$ with $q_{i,n}\cap |\gamma|\neq \emptyset$ we have
    $$ \lambda_{i,n}\leq \int_{\gamma} \frac{\lambda_{i,n}}{\diam q_{i,n}} \chi_{2q_{i,n}}\, ds.$$
    It follows that
    $$\sum_{\substack{i\in J_n\setminus K_n\\ q_{i,n}\cap|\gamma|\neq \emptyset}} \lambda_{i,n} \leq \int_{\gamma}\phi_n\, ds.$$
    Also, 
    $$ \sum_{\substack{i\in  K_n\\ q_{i,n}\cap|\gamma|\neq \emptyset}}\lambda_{i,n}\leq N \sup_{i\in I_n} \lambda_{i,n}.$$
    By \eqref{lemma:sum_tail:sup} and \eqref{lemma:sum_tail:integral}, we conclude that 
    $$\lim_{n\to\infty}\sum_{\substack{i\in J_n\\ q_{i,n}\cap|\gamma|\neq \emptyset}}\lambda_{i,n}=0 \quad \big(\textrm{resp. } \sum_{i: q_{i,n}\cap |\gamma|\neq \emptyset}\lambda_{i,n} <\infty\big)$$
    whenever $\gamma\notin \Gamma_1$ and $\gamma$ is non-constant. 

    Finally, note that the family $\Gamma_2$ of non-constant curves passing through the countably many points $q_{i,n}$, $i\in I_n\setminus J_n$, has $2$-modulus zero \cite{Vaisala:quasiconformal}*{\S 7.9}. This implies the statement for $\Gamma_0=\Gamma_1\cup \Gamma_2$. 
\end{proof}

In the next lemma, $g_j=f_j^{-1}:D_j \to \Omega_j$, $j=1,2,\ldots$, are the maps defined in the beginning of this section. 

\begin{lemma}\label{lemma:sum_limsup_g}
    If we pass to a subsequence of $(g_j)_{j \in \N}$, there exists a family of curves $\Gamma_0$ in $\hatc$ with $\modu\Gamma_0=0$ such that for all curves $\gamma\notin \Gamma_0$ we have 
    \begin{align*}
        \limsup_{j\to\infty} \sum_{\substack{q\in\mathcal C(D_j)\\ q\cap |\gamma|\neq \emptyset}}\diam 
        \hat g_j(q)\leq \sum_{\substack{q\in\mathcal C(D)\\ q\cap |\gamma|\neq \emptyset}} \diam \hat g(q).
    \end{align*}
\end{lemma}
    \begin{proof}
        We enumerate the components of $\mathcal C_N(\Omega)$ as $p_i$, $i\in I$, where $I\subset \N$, and let $q_i\in \mathcal C(D)$ be such that $\hat g(q_i)=p_i$. Let $n\in \mathbb N$ and consider a large enough $j(n)\geq n$ so that for $i\in I\cap \{1,\dots,n\}$ there are components $p_{i,j(n)} \in \mathcal C(\Omega_{j(n)})$ that are pairwise disjoint and 
        \begin{align}\label{lemma:sum_limsup_g:inclusions}
             p_{i} \subset p_{i,j(n)} \subset (1+1/n)p_i .
        \end{align}
        Consider $q_{i,n}\in \mathcal C(D_{j(n)})$ such that $\hat g_{j(n)}( q_{i,n})=p_{i,j(n)}$. 

        For each $i\in I$, by Theorem \ref{theorem:convergence}, $q_{i,n}$ converges to $q_i$ as $n\to\infty$. By Lemma \ref{lemma:sum_limsup}, we have
        \begin{align*}
        \limsup_{n\to\infty} \sum_{i:q_{i,n}\cap |\gamma|\neq \emptyset}\diam \hat g(q_i)\leq \sum_{i:q_{i}\cap |\gamma|\neq \emptyset} \diam \hat g(q_i) \leq \sum_{\substack{q\in\mathcal C(D)\\ q\cap |\gamma|\neq \emptyset}} \diam \hat g(q)
        \end{align*}
        for all curves $\gamma$ outside a family $\Gamma_0$ of $2$-modulus zero. Also from \eqref{lemma:sum_limsup_g:inclusions}, we obtain
        $$\diam \hat g_{j(n)}(q_{i,n})=\diam p_{i,j(n)} \leq (1+1/n)\diam \hat g(q_i),$$
        so we have
        \begin{align*}
        \limsup_{n\to\infty} \sum_{i:q_{i,n}\cap |\gamma|\neq \emptyset}\diam \hat g_{j(n)}(q_{i,n})\leq \sum_{\substack{q\in\mathcal C(D)\\ q\cap |\gamma|\neq \emptyset}} \diam \hat g(q).
        \end{align*}

        It remains to treat the sum of $\lambda_{n}(q)=\diam \hat g_{j(n)}(q)$ over $q\in I_n=\mathcal C(D_{j(n)})\setminus \{q_{i,n}: i\in I\cap  \{1,\dots,n\}\}$ with $q\cap |\gamma|\neq \emptyset$, and show that the limit is zero for all curves $\gamma$ outside another exceptional family of $2$-modulus zero. This is an immediate consequence of Lemma \ref{lemma:sum_tail}, upon verifying that 
        $$ \lim_{n\to\infty}\sum_{q\in I_n} \lambda_{n}(q)^2=0.$$
        Since $\partial \Omega$ has area zero, for each $\varepsilon>0$ there exists $\delta>0$ such that $$\Sigma(N_{\delta}(\partial \Omega)) < \varepsilon.$$
        We claim that there exists $N\in \N$ such that for $n>N$ and $q\in I_n$ we have $\hat g_{j(n)}(q) \subset N_{\delta}(\partial\Omega)$. Assuming this, since $(\Omega_j)_{j \in \N}$ is a quasiround exhaustion of $\Omega$, we have
        $$\sum_{q\in I_n} \lambda_n(q)^2 \leq C \sum_{q\in I_n} \Sigma( \hat g_{j(n)}(q)) \leq C \Sigma(N_{\delta}(\partial \Omega)) <C\varepsilon$$
        for $n>N$. This completes the proof. 
        
        Now we prove the claim. Since $(\Omega_{j})_{j \in \N}$ is an exhaustion of $\Omega$, there exists $N_0\in \N$ such that $\hat g_{j(n)}(q) \subset N_{\delta}(\hatc\setminus \Omega)$ for $q\in \mathcal C(D_{j(n)})$ and $n>N_0$. In particular, $$\hat g_{j(n)}(q) \cap \overline{\Omega} \subset N_{\delta}(\partial \Omega) \quad \textrm{for $q\in \mathcal C(D_{j(n)})$ and $n>N_0$}.$$
        Also, since $\Omega$ is a circle domain, there exists $N>N_0$ such that if $i\in I$ and $i>N$, then $\diam p_i<\delta$. If $q\in I_n$, then $\hat g_{j(n)}(q)$ does not intersect $p_1,\dots,p_n$. Hence, if $n>N$ and $q\in I_n$, then each point $z \in \hat g_{j(n)}(q)\setminus \overline{\Omega}$ lies in some $p_{i(z)}$ with $i(z)>n>N$. By our choice of $N$ we have $\diam p_{i(z)} <\delta$, and therefore $p_{i(z)}\subset N_{\delta}(\partial p_{i(z)})\subset N_{\delta}(\partial \Omega)$. This completes the proof of the claim.
    \end{proof}

\begin{proof}[Proof of Proposition \ref{proposition:upper_gradient}]
    Recall that $g_j=f_j^{-1}\colon D_j\to \Omega_j$, $j\in \N$. Since $g_j$ is a conformal map between finitely connected domains, we have
    \begin{align}\label{proposition:upper_gradient:finite}
        \sigma(g_j(\gamma(a)),g_j(\gamma(b)))\leq \int_{\gamma} |Dg_j|\chi_{D_j}\, ds+ \sum_{\substack{q\in \mathcal C(D_j) \\ q\cap |\gamma|\neq \emptyset}}\diam \hat g_j(q)
    \end{align}
    for every rectifiable curve $\gamma\colon [a,b]\to \hatc$ with $\gamma(a),\gamma(b)\in D_j$.  
    
    Since $g_j\to g$ locally uniformly in $D$, we have $|Dg_j|\to |Dg|$ locally uniformly in $D$. In fact, $|Dg_j|\chi_{D_j}$ also converges strongly in $L^2(\hatc)$ to $|Dg|\chi_D$. To see this, let $\varepsilon>0$ and $K\subset \Omega$ be a compact set with $\Sigma(\Omega\setminus K)<\varepsilon$. Then by kernel convergence, $f_j(K)$ is contained in a compact set $K'$ that is contained in $D_j$ for all large $j$. We have that $|Dg_j|\chi_{K'}\to |Dg|\chi_{K'}$ strongly in $L^2(\hatc)$ and 
    $$ \int |Dg_j|^2\chi_{D_j\setminus K'} \, d\Sigma \leq \Sigma(\Omega_j\setminus K)\leq \Sigma( \Omega\setminus K)<\varepsilon$$
    for large $j$. The integral of $|Dg|^2\chi_{D\setminus K'}$ is also less than $\varepsilon$. This implies the claim regarding strong convergence.   
    
    By Fuglede's lemma \cite{HeinonenKoskelaShanmugalingamTyson:Sobolev}*{p.\ 131}, there exists a curve family $\Gamma_1$ of $2$-modulus zero such that for $\gamma\notin \Gamma_1$ we have
    $$\int_{\gamma} |Dg_j|\chi_{D_j}\, ds \to \int_{\gamma} |Dg|\chi_D\, ds.$$
    By Lemma \ref{lemma:sum_limsup_g}, if we pass to a subsequence, there exists a curve family $\Gamma_2$ of $2$-modulus zero such that for $\gamma\notin \Gamma_2$ we have  
    \begin{align*}
        \limsup_{j\to\infty} \sum_{\substack{q\in\mathcal C(D_j)\\ q\cap |\gamma|\neq \emptyset}}\diam \hat g_j(q)\leq \sum_{\substack{q\in\mathcal C(D)\\ q\cap |\gamma|\neq \emptyset}} \diam \hat g(q).
    \end{align*}
    For $\gamma\notin \Gamma_0=\Gamma_1\cup \Gamma_2$ we combine the above with \eqref{proposition:upper_gradient:finite} to obtain the desired conclusion. 
\end{proof}

\section{Conformal extension to the sphere}\label{section:uniqueness}
In this section we complete the proof of Theorem \ref{mainthm} by proving the following more general result.

\begin{theorem}\label{theorem:uniqueness}
     Let $g\colon D\to \Omega$ be a conformal homeomorphism between circle domains $D,\Omega\subset \hatc$ that satisfies the conclusion of Proposition \ref{proposition:upper_gradient}. Then $g$ is the restriction of a M\"obius transformation of $\hatc$.
\end{theorem}

The proof closely follows the proof of \cite{Ntalampekos:rigidity_cned}*{Theorem 1.2}. The first step is to extend $g$ to a homeomorphism between the closures of $D$ and $\Omega$. This is facilitated by the transboundary upper gradient inequality of Proposition \ref{proposition:upper_gradient}. Then, via repeated reflections, $g$ can be extended to a homeomorphism of the sphere. The goal is to show that this extension is conformal, and thus a M\"obius transformation, with the aid of Theorem \ref{theorem:eccentric}. Each point of the sphere is either of interior type (i.e., lying in a reflected copy of the domain $D$) or of boundary type (i.e., lying in a reflected copy of the boundary $\partial D$) or of buried type (i.e., none of the above). In points of interior or buried type, it is easy to show that the eccentric distortion of $g$ is $1$, as required by Theorem \ref{theorem:eccentric}. Finally, we show that points of boundary type are small in the sense required by Theorem \ref{theorem:eccentric}, by applying again the transboundary upper gradient inequality of Proposition \ref{proposition:upper_gradient}. We conclude that $g$ has eccentric distortion $1$ except possibly in a small exceptional set and is therefore conformal.

We now initiate the proof. Let $g\colon D\to \Omega$ be a map as in Theorem \ref{theorem:uniqueness}. We establish some preliminary lemmas. 

\begin{lemma}\label{lemma:extension}
    The map $g$ extends to a homeomorphism from $\overline D$ onto $\overline \Omega$. 
\end{lemma}
\begin{proof}
    The conclusion of Proposition \ref{proposition:upper_gradient} is exactly the same as the conclusion of \cite{Ntalampekos:rigidity_cned}*{Theorem 3.1}. This conclusion is the main assumption for the considerations in \cite{Ntalampekos:rigidity_cned}*{Section 4}, which imply that $g$ is a \textit{packing-conformal map} in the sense of \cite{Ntalampekos:uniformization_packing}. As shown in  \cite{Ntalampekos:rigidity_cned}*{Section 4} this implies that $g$ has a homeomorphic extension to the closures; see also Theorems 6.1 and 7.1 in \cite{Ntalampekos:uniformization_packing}. 
\end{proof}

\begin{lemma}\label{lemma:length_zero}
    For each (anti-)M\"obius transformation $T\colon \hatc\to \hatc$ and for a.e.\ horizontal and a.e.\ vertical line segment $\gamma$ in $\mathbb C$ we have
    $$ \mathcal H^1( (g\circ T^{-1})(|\gamma| \cap T(\partial D)) )=0.$$
\end{lemma}

\begin{proof}
    By Lemma \ref{lemma:extension}, $g$ has a homeomorphic extension to the closure of $D$. This implies that 
    \begin{align}\label{lemma:length_zero:diam_zero}
        \textrm{$\diam \hat g(q)=0$ for $q\in \mathcal C(D)\setminus \mathcal C_N(D)$.}
    \end{align}
    Let $\Gamma_0$ be the curve family with $\modu\Gamma_0=0$ given by Proposition \ref{proposition:upper_gradient}. Then $\modu T(\Gamma_0)=0$. Therefore, for a.e.\ horizontal (resp.\ vertical) line segment $\gamma\colon [a,b]\to \mathbb C$ (with an injective parametrization), the inequality in Proposition \ref{proposition:upper_gradient} holds for $\widetilde \gamma=T^{-1}\circ \gamma$ and for all of its subcurves.  Moreover, by the very last inequality of Lemma \ref{lemma:sum_tail}, and upon enlarging the exceptional curve family $\Gamma_0$ if necessary, we may assume that the right-hand side of the inequality of Proposition \ref{proposition:upper_gradient} is finite for such $\widetilde \gamma$. Finally, if we enlarge again $\Gamma_0$, we may have $\mathcal H^1( |\widetilde \gamma|\cap \partial q)=0$ for each $q\in \mathcal C(D)$. In particular, the circular arc $|\widetilde \gamma|$ intersects each circle $\partial q$  in at most two points.

    Let $A\subset \mathcal C_N(D)$ be a finite set, and $B\subset D$ be a compact set. The set $[a,b]\setminus \widetilde \gamma^{-1}(B\cup  \bigcup\{q: q\in A \})$ is a countable union of disjoint intervals $O_j$, $j\in J$. Let $\gamma_j=\widetilde \gamma|_{O_j}$, $j\in J$, and observe that
    $$|\widetilde \gamma|\cap \partial D \subset \bigg(\bigcup_{j\in J} |\gamma_j|\cap \partial D\bigg) \cup \bigg(\bigcup_{q\in A} \partial q \cap |\widetilde \gamma| \bigg),$$
    where the latter union is a finite set. Therefore, 
    $$\mathcal H^1_{\infty}(g(|\widetilde \gamma|\cap \partial D)) \leq \sum_{j\in J}  \diam g( |\gamma_j| \cap \partial D).$$
    Note that the curves $\gamma_j$, $j\in J$, are pairwise disjoint subarcs of a circle in $\hatc$ and each disk $q\in \mathcal C_N(D)$ can intersect at most one of them.  Applying the inequality of Proposition \ref{proposition:upper_gradient} to each $\gamma_j$, $j\in J$, and using \eqref{lemma:length_zero:diam_zero}, we obtain 
    \begin{align*}
            \mathcal H^1_{\infty}(g(|\widetilde \gamma|\cap \partial D)) 
            &\leq \sum_{j\in J}\bigg(\int_{\gamma_j} |Dg|\chi_{D}\, ds+ \sum_{\substack{q\in \mathcal C_N(D) \\q\cap |\gamma_j|\neq \emptyset}}\diam \hat g(q)\bigg)\\
            &\leq \int_{\widetilde \gamma} |Dg|\chi_{D \setminus B}\, ds+ \sum_{\substack{q\in \mathcal C_N(D)\setminus A \\q\cap |\widetilde\gamma|\neq \emptyset}}\diam \hat g(q).
    \end{align*}
    As the set $A$ increases to $\mathcal C_N(D)$ and the set $B$ increases to $D$, the right-hand side converges to zero by dominated convergence. Hence 
    $$0=\mathcal H^1_{\infty}(g(|\widetilde \gamma|\cap \partial D))=\mathcal H^1(g(|\widetilde \gamma|\cap \partial D)) =\mathcal H^1((g\circ T^{-1})( |\gamma|\cap T(\partial D))).$$
    This completes the proof. 
\end{proof}

\begin{proof}[Proof of Theorem \ref{theorem:uniqueness}]
Without loss of generality we assume that $\infty \in D$ and $g(\infty)=\infty$. With the aid of Lemma \ref{lemma:extension}, we extend $g$ to a homeomorphism of $\hatc$ through reflections across the boundary circles of $D$. A detailed proof can be found in \cite{NtaYou20}*{Section 7.1}. Here we highlight the important features of the extension procedure.  

We denote by $S_i$, ${i\in I}$, the collection of circles in $\partial D$, by $B_i\subset \hatc\setminus \overline D$ the open ball bounded by $S_i$, and by $R_i$ the reflection across the circle $S_i$, $i\in I$. Here, we regard $I$ as a subset of $\mathbb N$.  Consider the free discrete group generated by the family of reflections $\{R_i: \, i\in I\}$. This is called the \textit{Schottky group} of $D$ and is denoted by $\Gamma(D)$. Each $T\in \Gamma(D)$ that is not the identity can be expressed uniquely as $T=R_{i_1}\circ \dots\circ R_{i_k}$, where $i_j\neq i_{j+1}$ for $j\in \{1,\dots, k-1\}$. We also note that $\Gamma(D)$ contains countably many elements.

By Lemma \ref{lemma:extension}, $g$ extends to a homeomorphism between $\overline D $ and $\overline \Omega$. Hence, there exists a natural bijection between $\Gamma(D)$ and $\Gamma(\Omega)$, induced by $g$. Namely, if $R_i^*$ is the reflection across the circle $S_i^*=g(S_i)$, then for $T=R_{i_1}\circ\dots\circ R_{i_k}$ we define $T^*=R_{i_1}^*\circ \dots \circ R_{i_k}^*$. By \cite{NtaYou20}*{Lemma 7.5}, there exists a unique extension of $g$ to a homeomorphism $\widetilde g$ of $\hatc$ with the property that $T^*=\widetilde g\circ T\circ \widetilde{g}^{-1}$ for each $T\in \Gamma(D)$. We will verify that $\widetilde g$ is conformal. For simplicity, we use the notation $g$ instead of $\widetilde g$. 

For each point $x\in \hatc$ we have the following trichotomy; see Lemma 7.2 and Corollary 7.4 in \cite{NtaYou20}.
\begin{enumerate}[label=\normalfont(\Roman*)]
	\item\label{type:interior} (Interior type) $x\in T(D)$ for some $T\in \Gamma(D)$.
	\item\label{type:boundary} (Boundary type) $x\in T(\partial D)$ for some $T\in \Gamma(D)$.
	\item\label{type:buried} (Buried type) There exists a sequence of indices $(i_j)_{j\in \mathbb N}$ with $i_j\neq i_{j+1}$ and disks $D_0=B_{i_1}$, $D_k=R_{i_1}\circ\dots\circ R_{i_k}(B_{i_{k+1}})$ such that $D_{k+1}\subset D_{k}$ for each $k\geq 0$ and $\{x\}=\bigcap_{k=0}^\infty D_k$.
\end{enumerate}
At each point $x$ of interior type \ref{type:interior}, the map $g$ is conformal, so it maps infinitesimal balls centered at $x$ to infinitesimal balls centered at $g(x)$. In particular, $E_g(x)=1$; recall the definition from Section \ref{section:definition_qc}. If  $x$ is of buried type \ref{type:buried}, then there exists a sequence of balls $D_k$, $k\in \mathbb N$, shrinking to $x$ such that $g(D_k)$, $k\in \mathbb N$, are balls shrinking to $g(x)$. It follows that $E_g(x)=1$.

Finally, we treat points of boundary type \ref{type:boundary}. By Lemma \ref{lemma:length_zero}, for each $T\in \Gamma(D)$ and for a.e.\ horizontal and a.e.\ vertical line segment $\gamma$ in $\mathbb C$, we have 
$$\mathcal H^1( (g\circ T^{-1})(|\gamma|\cap T(\partial D)))=0.$$
Since $(T^{-1})^*\circ g=g\circ T^{-1}$, and $(T^{-1})^*$ is bi-Lipschitz, we obtain
$$\mathcal H^1( g(|\gamma|\cap T(\partial D))) =\mathcal H^1( ((T^{-1})^*\circ g)(|\gamma|\cap T(\partial D)))=0.$$
Therefore, for a.e.\ horizontal and a.e.\ vertical line segment $\gamma$ in $\mathbb C$ we have
$$\mathcal H^1( g(|\gamma| \cap G)) =0$$
where $G$ is the countable union $\bigcup_{T\in \Gamma(D)} T(\partial D)$. Observe that the Euclidean metric is comparable to the spherical metric on $G$ so the above statement is true with respect to either metric. By Theorem \ref{theorem:eccentric}, since $G$ is measurable, we conclude that $g$ is conformal in $\hatc$.
\end{proof}

\section{No infinitesimally round exhaustions}\label{section:example}

It may be desirable to have an improvement of Theorem \ref{theorem:quasiround} in which quasi\-round exhaustions of a circle domain $\Omega\subset \hatc$ are replaced with \textit{infinitesimally round} exhaustions $(\Omega_j)_{j \in \N}$ for which each $\Omega_j$ is $K_j$-quasiround and $K_j \to 1$ as $j \to \infty$. However, such an improvement does not hold in general. We construct circle domains which do not admit infinitesimally round exhaustions. 

Consider the closed disk centered at $(0,2)$ with radius $1$ and the closed disk centered at $(0,2/3)$ with radius $1/3$. The two disks are tangent to each other, and we call them \textit{central}. We also consider two \textit{lateral} disks of radius $2/9$ that are tangent to both central disks, as shown in Figure \ref{fig:round}. We consider scaled copies of all these disks so that we obtain a packing as in the figure, with central disks of radii $\dots,9,3,1,1/3,1/9,\dots$ and lateral disks of radii $\dots,6,2,2/3,2/9,2/27,\dots$. The choice of the radii guarantees that all disks are contained in the region $|x|<y$. We rotate this packing by multiples of $\pi/2$ so that we obtain four disjoint packings, each contained in a complementary region of the lines $y=|x|$.

\begin{figure}
    \centering
    \begin{tikzpicture}[scale=1.5]

\def\x{0};
\def\y{2};

\draw (\x,\y) circle (1); \draw[dotted] (\x,\y) circle (1-1/6);
\fill (\x,\y) circle (1pt) node [anchor=south] {$(0,2)$};
\draw (0,2)-- (1,2); 
\node[anchor=south] at (0.5,2) {$1$};
\draw[dotted] (0,2)--(-0.77,1.6);
\node[anchor=west,scale=0.7] at(-0.5,1.7) {$1-k^{-1}$}; 

\draw (\x,\y/3) circle (1/3) node {$1/3$};\draw[dotted] (\x,\y/3) circle (1/3-1/12);
\draw (\x,\y/9) circle (1/9);
\draw (\x,\y/27) circle (1/27);

\draw (4.64/3,5.3/2) circle (2/3) node {$2/3$};\draw[dotted] (4.64/3,5.3/2) circle (2/3-1/6);
\draw (-4.64/3,5.3/2) circle (2/3);\draw[dotted] (-4.64/3,5.3/2) circle (2/3-1/6);
\draw (4.64/9,5.3/6) circle (2/9) node {$2/9$};\draw[dotted] (4.64/9,5.3/6) circle (2/9-1/18);
\draw (-4.64/9,5.3/6) circle (2/9);\draw[dotted] (-4.64/9,5.3/6) circle (2/9-1/18);
\draw (4.64/27,5.3/18) circle (2/27);
\draw (-4.64/27,5.3/18) circle (2/27);

\draw[dotted] (0,0) -- (3,3);
\draw[dotted] (0,0) -- (-3,3);

\fill (0,0) circle (1pt) node[anchor=west] {$(0,0)$};
\end{tikzpicture}
    \caption{Construction of the domain $G_k$.}
    \label{fig:round}
\end{figure}
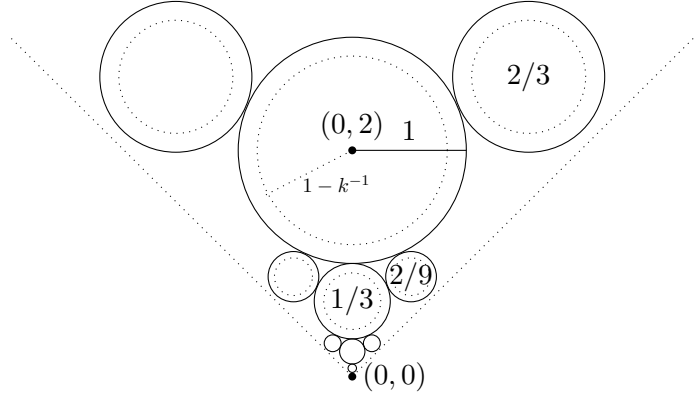

We denote by $D_m$, $m \in \mathbb{Z}$, the family of all the closed disks chosen above. Moreover, let 
\begin{eqnarray*} 
G_k &=& \mathbb{C} \setminus \Big(\{(0,0)\} \cup \bigcup_{m \in \mathbb{Z}}(1-k^{-1})D_m \Big), 
\quad k=1,2,\ldots, \\ 
G_\infty &=& \mathbb{C} \setminus \Big(\{(0,0)\} \cup \bigcup_{m \in \mathbb{Z}} D_m \Big). 
\end{eqnarray*}
Recall that if $D=\overline{\mathbb{D}}(a,r)$ then $\delta D=\overline{\mathbb{D}}(a,\delta r)$. See the dotted disks in Figure \ref{fig:round} for the construction of the domains $G_k$ and the larger disks for the construction of $G_\infty$ The complementary components of $G_k$ in $\mathbb{C}$ are by definition the point $(0,0)$ and the disks $(1-k^{-1})D_m$, so $G_k$ is a circle domain.  

We claim that there exists $k_0 \in \mathbb{N}$ such that if $k \geq k_0$ then the circle domain $G_k$ does not have infinitesimally round exhaustions. Suppose for the sake of contradiction that there is a subsequence $(G_{k_\ell})_{\ell\in \N} \eqqcolon(U_\ell)_{\ell\in \N}$ of $(G_k)_{k\in \N}$ such that each $U_\ell$ has an infinitesimally round exhaustion. It follows that for each $\ell\in \N$ there is a closed disk $\overline{\mathbb{D}}(a_\ell,R_\ell)$ and a Jordan curve $\tilde{J}_\ell \subset U_\ell$ separating $(0,0)$ from $\infty$ that bounds a Jordan region $W_\ell$ so that 
\begin{equation} \label{smallneighbor} 
0< R_\ell < \ell^{-1} \quad \text{and} \quad \overline{\mathbb{D}}(a_\ell,R_\ell) \subset W_{\ell}\subset  \overline{\mathbb D}(a_\ell, (1+\ell^{-1})R_\ell).
\end{equation}
We scale each $\tilde{J}_\ell$ by $3^{s_\ell}$, $s_\ell \in \mathbb{N}$,  so that the diameter of the scaled curve $J_\ell$ satisfies
$1/3 < \diam J_\ell \leq 1$ and $J_\ell$ separates $(0,0)$ from $\infty$. Notice that $J_\ell \subset U_\ell$ by the scaling invariance of $U_\ell$.

After possibly taking a subsequence, the curves $J_\ell$ converge in the Hausdorff sense. By \eqref{smallneighbor}, the limit is a circle that we denote by $\mathbb{S}_\infty$ and satisfies $1/3 \leq \diam \mathbb{S}_\infty \leq 1$. Since each $J_\ell$ surrounds the origin, we have $(0,0) \in D_\infty$, where $D_\infty$ is the closed disk bounded by $\mathbb{S}_\infty$. Moreover, since each $J_\ell$ is a subset of $U_\ell$, we have $\mathbb{S}_\infty \subset \overline{G_\infty}$. However, by construction $\overline{G_\infty}$ does not contain any circle that surrounds the origin; we leave this as an exercise to the reader. Therefore, we obtain a contradiction.
\bibliography{ExhaustionBiblio} 
\end{document}